\documentclass[11pt]{article}
\usepackage[tbtags]{amsmath}
\usepackage{amssymb}
\usepackage{amsthm}
\usepackage[misc]{ifsym}
\usepackage{cases}
\usepackage{mathrsfs}
\usepackage{stmaryrd}
\usepackage[colorlinks,linkcolor=black,anchorcolor=black,citecolor=black]{hyperref}
\usepackage{CJK}

\numberwithin{equation}{section}
\normalsize

\title{\bf A General Maximum Principle for Progressive Optimal Control of Fully Coupled Forward-Backward Stochastic Systems with Jumps \thanks{This work is supported by National Natural Science Foundations of China (12471419, 12271304), and Shandong Provincial Natural Science Foundations (ZR2024ZD35, ZR2022JQ01).}}

\author{\normalsize Bin Wang\thanks{\it School of Mathematics, Shandong University, Jinan 250100, P.R. China, E-mail: 202112005@mail.sdu.edu.cn} ,\quad Yu Si\thanks{\it School of Mathematics, Shandong University, Jinan 250100, P.R. China, E-mail: 202112003@mail.sdu.edu.cn} ,\quad Jingtao Shi\thanks{Corresponding author, \it School of Mathematics, Shandong University, Jinan 250100, P.R. China, E-mail: shijingtao@sdu.edu.cn}}

\date{}

\newtheorem{mypro}{Proposition}[section]
\newtheorem{mythm}{Theorem}[section]

\newtheorem{mylem}{Lemma}[section]
\newtheorem{myrem}{Remark}[section]

\begin{document}

\maketitle

\noindent{\bf Abstract:}\quad This paper is concerned with a general maximum principle for the fully coupled forward-backward stochastic optimal control problem with jumps, where the control domain is not necessarily convex, within the progressively measurable framework. A distinct feature in this paper is that the solution $Z$ of BSDEPs could include the variable ``$e$'', further, the diffusion term of BSDEPs takes the form $\int_{\mathcal{E}}Z_{(t,e)}\nu(d e)d W_t$ rather than the conventional $Z_t dW_t$, reflecting the essential coupling between the solution component $Z$ and the Polish space $\mathcal{E}$.

\vspace{2mm}


\section{\bf Introduction}

Pardoux and Peng are the pioneers to introduce the concept of a nonlinear {\it backward stochastic differential equation} (BSDE) in their paper \cite{PardouxPeng1990}. And then Peng \cite{Peng1990} obtained the general maximum principle for the stochastic optimal control problem. Here  the word {\it general} means that the control domain is not necessarily convex and the diffusion term is control dependent. Duffie and Epstein \cite{Duffie1992} pioneered of recursive utilities in continuous time, which is a specific type of BSDE. Recursive utility is an extension of the standard additive utility, where the instantaneous utility not only depend on the instantaneous consumption rate but also on the future utility. The {\it forward-backward stochastic differential equation} (FBSDE) was first introduced by Antonelli \cite{Antonelli1993}, and the problem of stochastic recursive optimal control is a type of forward-backward optimal control problems. By new variational technique, Hu \cite{Hu2017} completely solves Peng's open problem in \cite{Peng1998} and get a general maximum principle for the stochastic recursive optimal control problem. Furthermore, Hu et al. \cite{HuJiXue2018} extended it to the system governed by a fully coupled FBSDE. Afterwards, Lin and Shi \cite{LinShi2023} generalized the cost functional of this problem to a general form, and then, they \cite{LinShi2025} also study a kind of risk-sensitive optimal control problem for fully coupled forward-backward stochastic systems.

Due to their wide applications in finance and economics, many scholars have done research for the optimal control problem of stochastic systems with Poisson jumps. Situ \cite{Situ1991} first obtained the general maximum principle for a system described by a {\it stochastic differential equation with Poisson jumps} (SDEP), but the jump coefficient is control variable independent. Tang and Li \cite{TangLi1994} proved the general maximum principle of controlled SDEPs, where the control variable enters all coefficients. Later, Song et al. \cite{SongTangWu2020} introduced a novel spike variation technique and derived the general maximum principle for SDEPs under the progressively measurable framework. Zheng and Shi \cite{ZhengShi2023} obtained a general maximum principle for partially observed progressive optimal control problem of {\it forward-backward stochastic differential equations with Poisson jumps} (FBSDEP). Recently, another version of the general maximum principle of FBSDEPs for the case of complete observation is obtained in Wang and Shi \cite{WangShi2024}.

In previous studies, many scholars have obtained a series of important results about general maximum principles of FBSDEs and FBSDEPs, by applying the Ekeland variational principle. For the optimal control problems without jumps, Yong \cite{Yong2010} studied a fully coupled FBSDE with mixed initial-terminal conditions, and Wu \cite{Wu2013} studied a recursive utility of FBSDEs. For the ones with jumps, Shi \cite{Shi2012} researched an FBSDEP system, with control independent diffusion and jump coefficients, and a fully coupled FBSDEP system was investigated by Shi \cite{Shi2012-}. However, a well-known shortcoming for utilizing the Ekeland variational principle is that some unknown parameters are unavoidably introduced, which are somewhat very difficult to determine.

Motivated by the above mentioned literature, in this paper, we study an optimal control problem of a fully coupled FBSDEP with recursive utility, where the control domain is not necessarily convex and the control variable enters all coefficients. The target is to obtain a general maximum principle. The distinction of this article lies in the fact that we do not use the Ekeland variational principle. Recently, Yang and Moon \cite{YangMoon2023} also studied this type of problems, without using the maximum principle, but instead proposed sufficient conditions by the dynamic programming principle.

As is well known, when exploring the optimal control problem with non-convex control domains, spike variational technique and high order $L^p$-estimate are necessary mathematical tools. For the $L^p$-estimate and existence and uniqueness of solutions to (F)BSDE(P)s, there are numerous theoretical achievements. Antonelli \cite{Antonelli1993} investigated the existence and uniqueness results of fully coupled FBSDEs within a sufficiently small terminal time $T$. Pardoux and Tang \cite{PardouxTang1999} extended these results based on a fixed-point approach. Using a decoupling random field method, Ma et al. \cite{MaProtterYong1994} introduced the well-known Four Step Scheme for the solvability of fully coupled FBSDEs. Hu and Peng \cite{HuPeng1995} and Peng and Wu \cite{PengWu1999} got the results under monotonicity conditions. For the problem with jumps, Tang and Li \cite{TangLi1994} studied the existence and uniqueness of $L^2$-solution for BSDEPs (see also Barles et al.\cite{BBP1997}). Quenez and Sulem \cite{QS2013} studied some properties of linear BSDEPs and gave an $L^p$-solution and $L^2$-estimate to such BSDEPs with general generators. In addition, Wu \cite{Wu1999} proved the existence and uniqueness results of $L^2$-solution for fully coupled FBSDEPs under monotonicity conditions. Li and Wei \cite{LiWei2014} proved the $L^2$-estimate for fully coupled FBSDEPs for any terminal time $T$ and the $L^p$-estimate within sufficiently small terminal time $T$ under monotonicity conditions. Without the monotonicity conditions, \cite{LiWei2014} proved the uniqueness $L^2$-solution for fully coupled FBSDEPs for small durations. Rencently, there are also some scholars who have conducted related research in this field. Xie and Yu \cite{XieYu2020} obtained an $L^p$-solution and its related $L^p$-estimate for coupled FBSDEs with random coefficients on small durations under the Lipschitz condition and monotonicity assumption. And then they \cite{XieYu2023} studied a coupled coupled linear FBSDE and get the $L^p$-result in the monotonicity framework over large time intervals. Meng and Yang \cite{MengYang2023} gave a positive result for the open problem proposed by Yong \cite{Yong2020}, which the unique $L^2$-solution of fully coupled FBSDEs is an $L^p$-solution under the usual Lipschitz (Lipschitz constant of Z in diffusion term should be sufficiently small). Zheng and Shi \cite{ZhengShi2023} also studied a problem and got an $L^p$-solution and $L^p$-estimate of the fully coupled FBSDEPs with the assumption that the jump diffusion term is independent of $(Z,\tilde{Z})$.

Inspired by the work of \cite{HuJiXue2018}, as well as the innovative spike variation technique introduced by \cite{SongTangWu2020} to subtract the jump part, and the novel $L^p$-estimate for decoupled FBSDEPs derived by Zheng and Shi \cite{ZhengShiarXiv2023}, in this paper, we successfully obtain the $L^p$-estimate of fully coupled FBSDEPs which is different from \cite{ZhengShiarXiv2023} and extend the general maximum principle to the framework of fully coupled FBSDEPs within the progressive measurable framework. This advancement signifies the expansion of our theoretical framework to handle complex systems involving both continuous variations and abrupt jumps, thereby enhancing its practical applicability. To our best knowledge, this is the first time that the solution $Z$ of BSDEPs incorporates the variable ``$e$'', which sets our work apart from the previous research. 

The structure of the remaining sections of this paper is outlined as follows. In Section 2, we formally state our problem and provide an $L^p$-estimate of FBSDEPs. Section 3 constitutes the pivotal content of this article, where we derive a general maximum principle. Finally, in Section 4, some concluding remarks are given. Some proofs are postponed to the Appendix.

{\it Notations.}\quad In this paper, $\mathbf{R}^n$ denotes the $n$-dimensional Euclidean space with norm $|\cdot|$, $D g, D^2 g$ denotes the gradient and the Hessian matrix of the differentiable function $g$, respectively, $\top$ appearing as a superscript denotes the transpose of a matrix, and $C>0$ denotes a generic constant which may take different values in different places.

\section{\bf Problem formulation and preliminaries}

\subsection{Problem formulation}

Let $T>0$ be fixed and $\mathbf{U} \subset \mathbf{R}^k$ be nonempty. Let $(\Omega,\mathcal{F},P)$ be a complete probability space and $(\mathcal{E}, \mathscr{B}(\mathcal{E}))$ be a Polish space with a $\sigma$-finite measure $\nu(\cdot)$ on $\mathcal{E}$. $\{W_t\}_{t\geqslant 0}$ is a one-dimensional standard Brownian motion defined on $(\Omega,\mathcal{F},P)$ over $[0,T]$ (with $W_0=0$, $P\text{-a.s.})$ and $\{N(d e, d t)\}_{t\geqslant 0}$ is a Poisson random measure on $\left(\mathbf{R}^+ \times \mathcal{E}, \mathscr{B}\left(\mathbf{R}^+ \right)\times \mathscr{B}\left(\mathcal{E}\right)\right)$ on $(\Omega,\mathcal{F},P)$ over $[0,T]$ and for any $E \in \mathscr{B}\left(\mathcal{E}\right)$, $\nu\left(E\right)<\infty$, then the compensated Poisson random measure is defined by $\tilde{N}(d e, d t)\triangleq N(d e, d t)-\nu(d e) d t$. The filtration $\left\{\mathcal{F}_t;0\leqslant t\leqslant T\right\}$ is generated as the following $\mathcal{F}_t\triangleq\mathcal{F}_t^{W} \vee \mathcal{F}_t^{N} \vee \mathcal{N}$, where $W, N$ are mutually independent under $P$, $\mathcal{F}_t^{W}, \mathcal{F}_t^{N}$ are the $P$-completed natural filtrations generated by $W, N$ respectively, and $\mathcal{N}$ denotes the totality of $P$-null sets. Let $\mathbb{E}$ denotes the expectation under the probability measure $P$.

We denote by $L_{\mathcal{F}}^p([0,T];\mathbf{R})$ the space of all $\mathbf{R}$-valued $\mathcal{F}_t$-adapted processes $\phi_\cdot$ such that $\mathbb{E} \big[\sup_{0\leqslant t\leqslant T}|\phi_t|^p\big] <\infty$, by $L_{\mathcal{F}}^{2,p}([0,T]\times\mathcal{E};\mathbf{R})$ the space of all $\mathbf{R}$-valued $\mathcal{F}_t$-predictable processes $\varphi_{(\cdot,\cdot)}$ such that $\mathbb{E}\big(\int_0^T\int_{\mathcal{E}} |\varphi_{(t,e)}|^2 \nu(d e)d t\big)^{\frac{p}{2}}<\infty$, by $F^{2,p}([0,T]\times\mathcal{E};\mathbf{R})$ the space of all $\mathbf{R}$-valued $\mathcal{F}_t$-predictable processes $\tilde{\varphi}_{(\cdot,\cdot)}$ such that $\mathbb{E}\big(\int_0^T \int_{\mathcal{E}} |\tilde{\varphi}_{(t,e)}|^2 N(d e,d t)\big)^{\frac{p}{2}}<\infty$, and by $\mathcal{L}^2(\mathcal{E},\mathcal{B}(\mathcal{E}),\nu;\mathbf{R})$ or $\mathcal{L}^2$ the set of square-integrable functions $f:\mathcal{E}\rightarrow\mathbf{R}$ such that $\left\|f\right\|^2\triangleq\int_{\mathcal{E}}|f_e|^2\nu(de)<\infty$.

For simplicity, we denote $\mathscr{M}^p[0,T]\triangleq L_{\mathcal{F}}^p([0,T];\mathbf{R})\times L_{\mathcal{F}}^p([0,T];\mathbf{R})\times L_{\mathcal{F}}^{2,p}([0,T]\times\mathcal{E};\mathbf{R})\times F^{2,p}([0,T]\times\mathcal{E};\mathbf{R})$ and $\mathscr{N}^p[0,T]\triangleq L_{\mathcal{F}}^p([0,T];\mathbf{R})\times L_{\mathcal{F}}^{2,p}([0,T];\mathbf{R})\times F^{2,p}([0,T]\times\mathcal{E};\mathbf{R})$.

We denote by $\mathcal{U}[0,T]$ the set of all admissible controls $u:[0,T]\times\Omega\rightarrow\mathbf{U}$ which is an $\{\mathcal{F}_t\}_{t\geqslant 0}$-progressively measurable process on $(\Omega,\mathcal{F},P)$ such that $\mathbb{E}\big[\sup_{0\leqslant t\leqslant T}|u_t|^p\big] <\infty, p>1$.

Consider the following fully coupled forward-backward stochastic control system with jumps:
\begin{equation}\label{fully coupled FBSDEP}
\left\{\begin{aligned}
dX_t&= b\big(t,X_t,Y_t,Z_{(t,e)},\tilde{Z}_{(t,e)},u_t \big)d t +\sigma\big(t,X_t,Y_t,Z_{(t,e)},\tilde{Z}_{(t,e)},u_t \big)d W_t\\
&\quad +\int_{\mathcal{E}}f\big(t,X_{t-},Y_{t-},Z_{(t,e)},\tilde{Z}_{(t,e)},u_t,e\big)\tilde{N}(d e,d t),\\
-dY_t&= g\big(t,X_t,Y_t,Z_{(t,e)},\tilde{Z}_{(t,e)},u_t \big)d t -\int_{\mathcal{E}}Z_{(t,e)}\nu(d e)d W_t-\int_{\mathcal{E}}\tilde{Z}_{(t,e)}\tilde{N}(d e,d t),\\
X_0&= x_0,\quad Y_T=\phi(X_T),
\end{aligned}\right.
\end{equation}
where $b,\sigma,g:[0,T]\times\mathbf{R}\times\mathbf{R}\times\mathcal{L}^2\times\mathcal{L}^2\times\mathbf{U}\rightarrow\mathbf{R}$ , $f:[0,T]\times\mathbf{R}\times\mathbf{R}\times\mathcal{L}^2\times\mathcal{L}^2\times\mathbf{U}\times\mathcal{E}\rightarrow\mathbf{R}$ and $\phi:\mathbf{R}\rightarrow\mathbf{R}$ are given functions.

\begin{myrem}
Our state equation is different from others by the variable $``e"$ enter into $Z_{(t,e)}$, one of ``martingale terms" of this BSDEP. The reason can be seen from the form of $K_1(t,e)$ that follows \eqref{explicit solution of K1K2}, where $K_1(t,e)$ is a part of the first-order variation of $Z_{(t,e)}$, arising basically from the fully coupled structure of forward-backward equations. In addition, for $\phi = b,\sigma,g$, $\phi \big(t,X_t,Y_t,Z_{(t,e)},\tilde{Z}_{(t,e)},u_t \big)$ in $\eqref{fully coupled FBSDEP}$ should be written as $\phi \big(t,X_t,Y_t,\int_{\mathcal{E}} Z_{(t,e)} \nu (d e),\int_{\mathcal{E}} \tilde{Z}_{(t,e)} \nu (d e),u_t \big)$. For convenience, we write it as $\eqref{fully coupled FBSDEP}$.
\end{myrem}

Given $u_\cdot\in\mathcal{U}[0,T]$, we introduce the cost functional $J(u_\cdot)=Y_0$, and our stochastic recursive optimal control problem with jumps is the following.

\noindent{\bf Problem (SROCPJ).} Minimize the cost functional $J(u_\cdot)$ over $\mathcal{U}[0,T]$:
\begin{equation}
\inf\limits_{u_\cdot\in\, \mathcal{U}[0,T]}J(u_\cdot).
\end{equation}

\subsection{$L^p$-estimate of fully coupled FBSDEPs}

We give the $L^p$-estimate of decoupled FBSDEPs in the first, which are the generalization of Theorem 3.17, Theorem 5.17 in Pardoux and Ra\c{s}canu \cite{PardouxRasanu2014} and Lemma A.1 in Hu et al. \cite{HuJiXue2018} to Possion jumps.
\begin{mylem}\label{lemma 2.1}
For any stochastic processes $\left(y_\cdot,z_{(\cdot,\cdot)},\tilde{z}_{(\cdot,\cdot)}\right)\in \mathscr{N}^p[0,T]$, consider the following FBSDEP:
\begin{equation}\label{decoupled FBSDEPs}
\left\{\begin{aligned}
dX_t=&\ b\left(t,X_t,y_t,z_{(t,e)},\tilde{z}_{(t,e)} \right)d t + \sigma\left(t,X_t,y_t,z_{(t,e)},\tilde{z}_{(t,e)} \right) d W_t\\
& +\int_{\mathcal{E}}f\left(t,X_{t-},y_{t-},z_{(t,e)},\tilde{z}_{(t,e)},e\right)\tilde{N}(d e,d t),\\
-dY_t=&\ g\big(t,X_t,Y_t,Z_{(t,e)},\tilde{Z}_{(t,e)} \big) d t -\int_{\mathcal{E}}Z_{(t,e)}\nu(d e)d W_t-\int_{\mathcal{E}}\tilde{Z}_{(t,e)}\tilde{N}(d e,d t),\\
X_0=&\ x_0,\quad Y_T=\phi(X_T),
\end{aligned}\right.
\end{equation}
where $b,\sigma,g:[0,T]\times\mathbf{R}\times\mathbf{R}\times\mathcal{L}^2\times\mathcal{L}^2\times\mathbf{U}\rightarrow\mathbf{R}$ , $f:[0,T]\times\mathbf{R}\times\mathbf{R}\times\mathcal{L}^2\times\mathcal{L}^2\times\mathbf{U}\times\mathcal{E}\rightarrow\mathbf{R}$ and $\phi:\mathbf{R}\rightarrow\mathbf{R}$. If the coefficients satisfy:

$(1)$ $b,\sigma,g$ are $\mathscr{G}\otimes\mathscr{B}(\mathbf{R})\otimes\mathscr{B}(\mathbf{R})\otimes\mathscr{B}(\mathbf{R})\otimes\mathscr{B}(\mathbf{R})/\mathscr{B}(\mathbf{R})$ measurable, $f$ is $\mathscr{G}\otimes\mathscr{B}(\mathbf{R})\otimes\mathscr{B}(\mathbf{R})\otimes\mathscr{B}(\mathbf{R})\otimes\mathscr{B}(\mathbf{R})\otimes\mathscr{B}(\mathcal{E})/\mathscr{B}(\mathbf{R})$ measurable and $\phi$ is $\mathscr{F}_T\otimes\mathscr{B}(\mathbf{R})/\mathscr{B}(\mathbf{R})$ measurable, where $\mathscr{G}(\mathscr{P})$ is the progressive (predictable) $\sigma$-field on $[0,T]\times\Omega;$

$(2)$ For $p\geqslant 2,t\in[0,T],e\in\mathcal{E}$, we have
\begin{equation*}
\begin{aligned}
&\mathbb{E}\Bigg\{\big|\phi(0)\big|^p+\left(\int_0^T \big|b(t,0,y_t,z_{(t,e)},\tilde{z}_{(t,e)})\big|d t\right)^p+\left(\int_0^T\big|g(t,0,0,0,0)\big|d t\right)^p \\
&\ +\left(\int_0^T\big|\sigma(t,0,y_t,z_{(t,e)},\tilde{z}_{(t,e)})\big|^2 d t\right)^{\frac{p}{2}}
 +\left(\int_0^T\int_{\mathcal{E}}\big|f(t,0,y_t,z_{(t,e)},\tilde{z}_{(t,e)},e)\big|^2 N(d e,d t)\right)^{\frac{p}{2}}\Bigg\}<\infty;
\end{aligned}
\end{equation*}

$(3)$ For $\psi=b,\sigma$,
\begin{equation*}
\begin{array}{c}
\big|\psi\left(t,x,y,z,\tilde{z}\right)-\psi\left(t,\hat{x},y,z,\tilde{z}\right)\big|\leqslant L_1\big|x-\hat{x}\big|, \\[2ex]
\big|g\big(t,x,y,z,\tilde{z}\big)-g\big(t,\hat{x},\hat{y},\hat{z},\hat{\tilde{z}}\big)\big|\leqslant
L_1\big|x-\hat{x}\big|+L_2 \left(\big|y-\hat{y}\big|+\big\|z-\hat{z}\big\|+\big\|\tilde{z}-\hat{\tilde{z}}\big\|\right),\\[2ex]
\big\|f\left(t,x,y,z,\tilde{z},e\right)-f\left(t,\hat{x},y,z,\tilde{z},e\right)\big\|\leqslant L_1\big|x-\hat{x}\big|, \\
\end{array}
\end{equation*}
for all $t \in [0,T], x, \hat{x} \in \mathbf{R}, y, \hat{y} \in \mathbf{R}, z, \hat{z} \in \mathcal{L}^2, \tilde{z}, \hat{\tilde{z}} \in  \mathcal{L}^2, e \in \mathcal{E}$.

Then \eqref{decoupled FBSDEPs} has a unique solution $\big(X_\cdot,Y_\cdot,Z_{(\cdot,\cdot)},\tilde{Z}_{(\cdot,\cdot)}\big)\in \mathscr{M}^p[0,T]$ and there exists a constant $C_p>0$ which only depends on $p,T,L_1,L_2$ such that
\begin{equation*}
\begin{aligned}
&\mathbb{E}\Bigg\{\sup\limits_{0\leqslant t\leqslant T}\left[\big|X_t\big|^p+\big|Y_t\big|^p\right]+\left(\int_0^T\big\|Z_{(t,e)}\big\|^2 d t\right)^{\frac{p}{2}}
 +\left(\int_0^T\int_{\mathcal{E}}\big|\tilde{Z}_{(t,e)}\big|^2N(d e,d t)\right)^{\frac{p}{2}}\Bigg\}\\
\leqslant&\ C_p\mathbb{E}\Bigg\{\big|\phi(0)\big|^p+\big|x_0\big|^p+\left(\int_0^T \big|b(t,0,y_t,z_{(t,e)},\tilde{z}_{(t,e)})\big|d t\right)^p+\left(\int_0^T\big|g(t,0,0,0,0)\big|d t\right)^p \\
&+\left(\int_0^T\big|\sigma(t,0,y_t,z_{(t,e)},\tilde{z}_{(t,e)})\big|^2 d t\right)^{\frac{p}{2}}+\left(\int_0^T\int_{\mathcal{E}}\big|f(t,0,y_t,z_{(t,e)},\tilde{z}_{(t,e)},e)\big|^2 N(d e,d t)\right)^{\frac{p}{2}}\Bigg\}.
\end{aligned}
\end{equation*}
\end{mylem}

\begin{proof}
We can obtain the result immediately from Theorem A.2 in Song and Wu \cite{SongWu2023} by $b^2,\sigma^2,c^2,x_0^2 \equiv 0$ and Lemma 2.2, Lemma 2.3 in Zheng and Shi \cite{ZhengShiarXiv2023} by $\xi^2,g^2,y^2,z^2,\tilde{z}^2\equiv 0$.
\end{proof}

Then we will consider the $L^p$-estimate of the following fully coupled FBSDEP:
\begin{equation}\label{fully coupled FBSDEPs}
\left\{\begin{aligned}
dX_t&= b\big(t,X_t,Y_t,Z_{(t,e)},\tilde{Z}_{(t,e)}\big)d t  +\sigma\big(t,X_t,Y_t,Z_{(t,e)},\tilde{Z}_{(t,e)}\big)d W_t\\
&\quad +\int_{\mathcal{E}}f\big(t,X_{t-},Y_{t-},Z_{(t,e)},\tilde{Z}_{(t,e)},e\big)\tilde{N}(d e,d t),\\
-dY_t&=g\big(t,X_t,Y_t,Z_{(t,e)},\tilde{Z}_{(t,e)}\big)d t -\int_{\mathcal{E}}Z_{(t,e)}\nu(d e)d W_t-\int_{\mathcal{E}}\tilde{Z}_{(t,e)}\tilde{N}(d e,d t),\\
X_0&= x_0,\quad Y_T=\phi(X_T),
\end{aligned}\right.
\end{equation}
where $b,\sigma,f,g,\phi$ are the same in \eqref{decoupled FBSDEPs}.

Firstly, we need the following assumption.

\noindent{\bf (H0)} (1) For $\psi=b,\sigma,g,\phi$, $\psi$ is uniformly continuous in $x,y,z,\tilde{z}$, and there exsit some constants $L_i>0$, such that
$$
\begin{aligned}
&\quad \big|\psi\big(t,x,y,z,\tilde{z}\big)-\psi\big(t,\hat{x},\hat{y},\hat{z},\hat{\tilde{z}}\big)\big|\leqslant L_1\big|x-\hat{x}\big|+L_2\big|y-\hat{y}\big|+L_3\big\|z-\hat{z}\big\|+L_4\big\|\tilde{z}-\hat{\tilde{z}}\big\|, \\
& \big\|f\big(t,x,y,z,\tilde{z},e\big)-f\big(t,\hat{x},\hat{y},\hat{z},\hat{\tilde{z}},e\big)\big\|\leqslant L_1\big|x-\hat{x}\big|+L_2\big|y-\hat{y}\big|+L_3\big\|z-\hat{z}\big\|+L_4\big\|\tilde{z}-\hat{\tilde{z}}\big\|.
\end{aligned}
$$

(2) For $p\geqslant 2,t\in[0,T],e\in\mathcal{E}$, we have
$$
\begin{aligned}
& \mathbb{E}\Bigg\{\big|\phi(0)\big|^p+\left(\int_0^T \big|b(t,0,0,0,0)\big|d t\right)^p +\left(\int_0^T \big|g(t,0,0,0,0)\big|d t\right)^p \\
&\quad +\left(\int_0^T\big|\sigma(t,0,0,0,0)\big|^2 d t\right)^{\frac{p}{2}}+\left(\int_0^T \int_{\mathcal{E}}\big|f(t,0,0,0,0,e)\big|^2 N(d e,d t)\right)^{\frac{p}{2}}\Bigg\}<\infty.
\end{aligned}
$$

(3) Suppose $\Lambda_p\triangleq 2\left(2p\right)^{\frac{p}{2}}C_p C_0^p(1+T^p)<1$, where $C_0=\max \{L_2,L_3,L_4\}$; $C_p$ is defined in Lemma 2.1.

(4) The jump coefficient $f$ is independent of $Z$.

\begin{mypro}\label{Lp-estimate of fully coupled FBSDEP}
Let {\bf(H0)} hold, then \eqref{fully coupled FBSDEPs} admits a unique solution $\big(X_\cdot,Y_\cdot,Z_{(\cdot,\cdot)},\tilde{Z}_{(\cdot,\cdot)}\big)\in \mathscr{M}^p[0,T]$, and there exists a constant $C>0$ which depends on $p,T,L_1,C_0$ such that
\begin{equation}
\begin{aligned}
&\mathbb{E}\Bigg\{\sup\limits_{0\leqslant t\leqslant T}\left[\big|X_t\big|^p+\big|Y_t\big|^p\right]+\left(\int_0^T\big\|Z_{(t,e)}\big\|^2 d t\right)^{\frac{p}{2}}
 +\left(\int_0^T\int_{\mathcal{E}}\big|\tilde{Z}_{(t,e)}\big|^2N(d e,d t)\right)^{\frac{p}{2}}\Bigg\}\\
\leqslant&\ C\mathbb{E}\Bigg\{\big|\phi(0)\big|^p+\big|x_0\big|^p+\left(\int_0^T \big|b(t,0,0,0,0)\big|d t\right)^p +\left(\int_0^T \big|g(t,0,0,0,0)\big|d t\right)^p \\
&\qquad +\left(\int_0^T\big|\sigma(t,0,0,0,0)\big|^2 d t\right)^{\frac{p}{2}}+\left(\int_0^T \int_{\mathcal{E}}\big|f(t,0,0,0,e)\big|^2 N(d e,d t)\right)^{\frac{p}{2}}\Bigg\}.
\end{aligned}
\end{equation}
\end{mypro}

The proof is in the Appendix.

\begin{myrem}
We obtain the existence and uniqueness of solutions to \eqref{fully coupled FBSDEPs} under {\bf(H0)}. In fact, by $C_p>1,2^{p+1}(1+T^p)>1$, the condition $\Lambda_p<1$ is equivalent to $C_0$ is sufficiently small, which implies that the Lipschitz constants of $(Y,Z,\tilde{Z})$ in drift, diffusion and jump terms are all sufficiently small. It is worth noting that our conditions here are different from the scenario where the time duration $T$ is sufficiently small. In our case, even with $T$ being made sufficiently small, one cannot achieve $\Lambda_p<1$. Instead, the only viable approach is to ensure that the Lipschitz constants are sufficiently small.
\end{myrem}

\section{\bf General maximum principle}

Inspired by \cite{HuJiXue2018}, we derive a general maximum principle for the {\bf Problem (SROCPJ)} in this section. For convenience, the constant $C$ will change from line to line in our proofs.

For $\psi=b,\sigma,g,\phi$ (the coefficients of \eqref{fully coupled FBSDEP}), we assume that

\noindent{\bf (H1)} (1) $\psi,\psi_x,\psi_y,\psi_z,\psi_{\tilde{z}}$ are continuous in $(x,y,z,\tilde{z})$; $\psi_x,\psi_y,\psi_z,\psi_{\tilde{z}}$ are bounded; $\psi_{x x},\psi_{x y},\psi_{x z}$, $\psi_{x \tilde{z}},\psi_{y y},\psi_{y z},\psi_{y \tilde{z}},\psi_{z z},\psi_{z \tilde{z}},\psi_{\tilde{z} \tilde{z}}$ are continuous in $(x,y,z,\tilde{z})$ and bounded; there exists a constant $C>0$ such that
$$
\begin{aligned}
&\big|\psi(t,x,y,z,\tilde{z},u)\big|\leqslant C\left(1+\big|x\big|+\big|y\big|+\big\|z\big\|+\big\|\tilde{z}\big\|+\big|u\big|\right), \\
&\big|\psi(t,0,0,z,\tilde{z},u)-\psi(t,0,0,z,\tilde{z},\hat{u})\big|\leqslant C\left(1+\big|u\big|+\big|\hat{u}\big|\right),
\end{aligned}
$$
for all $t\in[0,T],x\in\mathbf{R},y\in\mathbf{R},z\in\mathcal{L}^2,\tilde{z}\in\mathcal{L}^2,u,\hat{u}\in \mathbf{U}$.

(2) $f,f_x,f_y,f_z,f_{\tilde{z}}$ are continuous in $(x,y,z,\tilde{z})$; $f_x,f_y,f_z,f_{\tilde{z}}$ are bounded; $f_{x x},f_{x y},f_{x z},f_{x \tilde{z}},f_{y y}$, $f_{y z},f_{y \tilde{z}},f_{z z},f_{z \tilde{z}},f_{\tilde{z} \tilde{z}}$ are continuous in $(x,y,z,\tilde{z})$ and bounded; there exists a constant $C>0$ such that
$$
\begin{aligned}
&\big\|f(t,x,y,z,\tilde{z},u,e)\big\|\leqslant C\left(1+\big|x\big|+\big|y\big|+\big\|z\big\|+\big\|\tilde{z}\big\|+\big|u\big|\right), \\
&\big\|f(t,0,0,z,\tilde{z},u,e)-f(t,0,0,z,\tilde{z},\hat{u},e)\big\|\leqslant C\left(1+\big|u\big|+\big|\hat{u}\big|\right),
\end{aligned}
$$
for all $t\in[0,T],x\in\mathbf{R},y\in\mathbf{R},z\in\mathcal{L}^2,\tilde{z}\in\mathcal{L}^2,u,\hat{u}\in \mathbf{U}$.

(3) Suppose $\Lambda_p\triangleq 2\left(2p\right)^{\frac{p}{2}}C_p C_0^p(1+T^p)<1$, where $C_0=\max \big\{L_2,L_3,L_4\big\}$, which are the Lipschitz constants of $(Y,Z,\tilde{Z})$ for $b,\sigma,f,g$, respectively, and $C_p$ is defined in Lemma \ref{lemma 2.1}.

(4) The jump coefficient $f$ is independent of $Z$.

For any $u_\cdot\in\mathcal{U}[0,T]$, under ${\bf (H1)}$, it is obvious that FBSDEP \eqref{fully coupled FBSDEP} admit a unique solution by Proposition \ref{Lp-estimate of fully coupled FBSDEP}.

Let $\bar{u}_\cdot$ be an optimal control, and $\big(\bar{X}_\cdot, \bar{Y}_\cdot, \bar{Z}_{(\cdot,\cdot)},\bar{\tilde{Z}}_{(\cdot,\cdot)}\big)\in \mathscr{M}^p[0,T]$ be the corresponding optimal trajectory of \eqref{fully coupled FBSDEP}. Since the control domain $\mathbf{U}$ is not necessarily convex, we resort to the new spike variation method introduced by \cite{SongTangWu2020}. For any $u_\cdot\in\mathcal{U}[0,T]$ and $\bar{t}\in[0,T],\epsilon>0$, define
$$
u_t^\epsilon\triangleq
\begin{cases}u_t, & \text { if }(t, \omega) \in \mathcal{O}\triangleq\ \rrbracket \bar{t}, \bar{t}+\epsilon \rrbracket \backslash \bigcup_{n=1}^{\infty} \llbracket T_n \rrbracket, \\
 \bar{u}_t, & \text { otherwise, }
\end{cases}
$$
where $\llbracket T_n \rrbracket\triangleq\left\{(\omega, t) \in \Omega \times[0, T] \mid T_n(\omega)=t\right\}$, the graph of stopping time $T_n$, is a progressive set, and $u_\cdot$ is a bounded $\mathcal{F}_{\bar{t}}$-measurable function that takes values in $\mathbf{U}$. Therefore, $u_\cdot^\epsilon$ is progressive and it is easy to show that $u_\cdot^\epsilon \in \mathcal{U}[0,T]$. Let $\big({X}_\cdot^\epsilon, {Y}_\cdot^\epsilon, {Z}_{(\cdot,\cdot)}^\epsilon,{\tilde{Z}}_{(\cdot,\cdot)}^\epsilon\big)$ be the corresponding trajectory of \eqref{fully coupled FBSDEP} associated with $u_\cdot^\epsilon$.

For simplicity, we denote
$$
\begin{aligned}
&\qquad\quad b(t)\triangleq b\big(t,\bar{X}_t,\bar{Y}_t,\bar{Z}_{(t,e)},\bar{\tilde{Z}}_{(t,e)},\bar{u}_t\big),\quad
b_x(t)\triangleq b_x\big(t,\bar{X}_t,\bar{Y}_t,\bar{Z}_{(t,e)},\bar{\tilde{Z}}_{(t,e)},\bar{u}_t\big), \\
&\delta b(t)\triangleq b\big(t,\bar{X}_t,\bar{Y}_t,\bar{Z}_{(t,e)},\bar{\tilde{Z}}_{(t,e)},u_t\big)-b(t), \quad
\delta b_x(t)\triangleq b_x\big(t,\bar{X}_t,\bar{Y}_t,\bar{Z}_{(t,e)},\bar{\tilde{Z}}_{(t,e)},u_t\big)-b_x(t), \\
&\qquad\qquad\quad \delta b(t,\Delta^1,\Delta^2)\triangleq b\big(t,\bar{X}_t,\bar{Y}_t,\bar{Z}_{(t,e)}+\Delta^1(t),\bar{\tilde{Z}}_{(t,e)}+\Delta^2(t),u_t\big)-b(t), \\
&\qquad\qquad \delta b_x(t,\Delta^1,\Delta^2)\triangleq b_x\big(t,\bar{X}_t,\bar{Y}_t,\bar{Z}_{(t,e)}+\Delta^1(t),\bar{\tilde{Z}}_{(t,e)}+\Delta^2(t),u_t\big)-b_x(t),
\end{aligned}
$$
and similar notations used for $b,\sigma,f,g,\phi$ and their derivatives of $x,y,z,\tilde{z}$, and $\Delta^i(\cdot),i=1,2$ are $\mathcal{F}_t$-adapted processes.

\begin{mylem}\label{estimates of X,Y,Z,tilde Z}
Let {\bf (H1)} hold. Then for $\beta\geqslant2$, we obtain that
\begin{equation}
\begin{aligned}
&\mathbb{E}\Bigg\{\sup\limits_{0\leqslant t\leqslant T}\left[|X_t^\epsilon-\bar{X}_t|^\beta+|Y_t^\epsilon-\bar{Y}_t|^\beta\right]
 +\left(\int_0^T \left\|Z_{(t,e)}^\epsilon-\bar{Z}_{(t,e)}\right\|^2 d t\right)^{\frac{\beta}{2}} \\
&\quad +\left(\int_0^T\int_{\mathcal{E}}\left|\tilde{Z}_{(t,e)}^\epsilon-\bar{\tilde{Z}}_{(t,e)}\right|^2 N(d e,d t)\right)^{\frac{\beta}{2}}\Bigg\}=O\left(\epsilon^{\frac{\beta}{2}}\right).
\end{aligned}
\end{equation}
\end{mylem}

The proof is in the Appendix.

\begin{myrem}
Notice that the estimate is derived by subtracting the jump part on $\mathcal{O}$, which means the jump term does not influence the order of variation (for more details see \cite{SongTangWu2020} and \cite{ZhengShi2023}). Because of this, the expansion of $f$ below is significantly different from other terms.
\end{myrem}

Therefore, by Lemma \ref{estimates of X,Y,Z,tilde Z}, we can set
\begin{equation}
\begin{array}{c}
X_t^\epsilon-\bar{X}_t=X_t^1+X_t^2+o(\epsilon),\qquad Y_t^\epsilon-\bar{Y}_t=Y_t^1+Y_t^2+o(\epsilon), \\[2ex]
Z_{(t,e)}^\epsilon-\bar{Z}_{(t,e)}=Z_{(t,e)}^1+Z_{(t,e)}^2+o(\epsilon),\qquad \tilde{Z}_{(t,e)}^\epsilon-\bar{\tilde{Z}}_{(t,e)}=\tilde{Z}_{(t,e)}^1+\tilde{Z}_{(t,e)}^2+o(\epsilon),
\end{array}
\end{equation}
where $X_t^1,Y_t^1,Z_{(t,e)}^1,\tilde{Z}_{(t,e)}^1\sim O\big(\epsilon^{\frac{1}{2}}\big)$ and $X_t^2,Y_t^2,Z_{(t,e)}^2,\tilde{Z}_{(t,e)}^2\sim O(\epsilon)$. Inspired by \cite{HuJiXue2018}, we also set
\begin{equation}\label{definition of Delta12}
\begin{array}{c}
Z_{(t,e)}^1=\dot{Z}_{(t,e)}^1+\Delta^1(t)I_{[\bar{t},\bar{t}+\epsilon]}(t),\qquad \tilde{Z}_{(t,e)}^1=\dot{\tilde{Z}}_{(t,e)}^1+\Delta^2(t)I_{\mathcal{O}}(t),
\end{array}
\end{equation}
where $\Delta^1(\cdot),\Delta^2(\cdot)$ are $\mathcal{F}_t$-adapted processes, and $\dot{Z}_{(t,e)}^1,\dot{\tilde{Z}}_{(t,e)}^1$ have good estimates similarly as $X_t^1$. So the expansions of $b,\sigma,f,g$ can be given as follows:
$$
\begin{aligned}
&\ b\big(t,X_t^\epsilon,Y_t^\epsilon,Z_{(t,e)}^\epsilon,\tilde{Z}_{(t,e)}^\epsilon,u_t^\epsilon\big)-b(t) \\
=&\ b\big(t,\bar{X}_t+X_t^1+X_t^2,\bar{Y}_t+Y_t^1+Y_t^2,\bar{Z}_{(t,e)}+\dot{Z}_{(t,e)}^1+\Delta^1(t)I_{[\bar{t},\bar{t}+\epsilon]}(t)+Z_{(t,e)}^2,\\
&\qquad \bar{\tilde{Z}}_{(t,e)}+\dot{\tilde{Z}}_{(t,e)}^1+\Delta^2(t)I_{\mathcal{O}}(t)+\tilde{Z}_{(t,e)}^2,u_t^\epsilon\big)-b(t)+ o(\epsilon) \\
=&\ b_x(t)\big(X_t^1+X_t^2\big)+b_y(t)\big(Y_t^1+Y_t^2\big)+b_z(t)\int_{\mathcal{E}}\big(\dot{Z}_{(t,e)}^1+Z_{(t,e)}^2\big)\nu(d e) \\
& +b_{\tilde{z}}(t)\int_{\mathcal{E}}\big(\dot{\tilde{Z}}_{(t,e)}^1+\tilde{Z}_{(t,e)}^2\big)\nu(d e) +\frac{1}{2}\Xi_tD^2b(t)\Xi_t^\top+\delta b\left(t,\Delta^1,\Delta^2\right)I_{[\bar{t},\bar{t}+\epsilon]}(t)+o(\epsilon), \\
&\ \sigma\big(t,X_t^\epsilon,Y_t^\epsilon,Z_{(t,e)}^\epsilon,\tilde{Z}_{(t,e)}^\epsilon,u_t^\epsilon\big)-\sigma(t) \\
=&\ \sigma\big(t,\bar{X}_t+X_t^1+X_t^2,\bar{Y}_t+Y_t^1+Y_t^2,\bar{Z}_{(t,e)}+\dot{Z}_{(t,e)}^1+\Delta^1(t)I_{[\bar{t},\bar{t}+\epsilon]}(t)+Z_{(t,e)}^2,\\
&\qquad \bar{\tilde{Z}}_{(t,e)}+\dot{\tilde{Z}}_{(t,e)}^1+\Delta^2(t)I_{\mathcal{O}}(t)+\tilde{Z}_{(t,e)}^2,u_t^\epsilon\big)-\sigma(t)+ o(\epsilon) \\
=&\ \sigma_x(t)\big(X_t^1+X_t^2\big)+\sigma_y(t)\big(Y_t^1+Y_t^2\big)+\sigma_z(t)\int_{\mathcal{E}}\big(\dot{Z}_{(t,e)}^1+Z_{(t,e)}^2\big)\nu(d e)\\
&\ +\sigma_{\tilde{z}}(t)\int_{\mathcal{E}}\big(\dot{\tilde{Z}}_{(t,e)}^1+\tilde{Z}_{(t,e)}^2\big)\nu(d e)+\delta\sigma_x\left(t,\Delta^1,\Delta^2\right)X_t^1I_{[\bar{t},\bar{t}+\epsilon]}(t) \\
&\ +\delta\sigma_y\left(t,\Delta^1,\Delta^2\right)Y_t^1I_{[\bar{t},\bar{t}+\epsilon]}(t)
 +\delta\sigma_z\left(t,\Delta^1,\Delta^2\right)\int_{\mathcal{E}}\dot{Z}_{(t,e)}^1\nu(d e)I_{[\bar{t},\bar{t}+\epsilon]}(t) \\
&\ +\delta\sigma_{\tilde{z}}\left(t,\Delta^1,\Delta^2\right)\int_{\mathcal{E}}\dot{\tilde{Z}}_{(t,e)}^1\nu(d e) I_{[\bar{t},\bar{t}+\epsilon]}(t)+\frac{1}{2}\Xi_tD^2\sigma(t)\Xi_t^\top+\delta \sigma\left(t,\Delta^1,\Delta^2\right)I_{[\bar{t},\bar{t}+\epsilon]}(t)+o(\epsilon), \\
\end{aligned}
$$
and
$$
\begin{aligned}
&\ f\big(t,X_t^\epsilon,Y_t^\epsilon,\tilde{Z}_{(t)}^\epsilon,u_t^\epsilon,e\big)-f(t,e) \\
=&\ f\big(t,\bar{X}_t+X_t^1+X_t^2,\bar{Y}_t+Y_t^1+Y_t^2, \bar{\tilde{Z}}_{(t,e)}+\dot{\tilde{Z}}_{(t,e)}^1+\Delta^2(t)I_{\mathcal{O}}(t)+\tilde{Z}_{(t,e)}^2,u_t^\epsilon,e\big)-f(t,e)+ o(\epsilon) \\
=&\ f_x(t,e)\big(X_t^1+X_t^2\big)+f_y(t,e)\big(Y_t^1+Y_t^2\big)+f_{\tilde{z}}(t,e)\big(\dot{\tilde{Z}}_{(t,e)}^1+\tilde{Z}_{(t,e)}^2\big)+\frac{1}{2}\Xi_tD^2f(t,e)\Xi_t^\top+o(\epsilon),\\
\end{aligned}
$$
where $\Xi_t\triangleq\left[X_t^1,Y_t^1,\int_{\mathcal{E}}\dot{Z}_{(t,e)}^1\nu(d e),\int_{\mathcal{E}}\dot{\tilde{Z}}_{(t,e)}^1\nu(d e)\right]$ and $g$ is similar as $b$.

First, we introduce the following two equations:
\begin{equation}\label{variational equation of X}
\left\{\begin{aligned}
&d\left(X_t^1+X_t^2\right)=\bigg\{b_x(t)\left(X_t^1+X_t^2\right)+b_y(t)\left(Y_t^1+Y_t^2\right)+b_z(t)\int_{\mathcal{E}}\big(\dot{Z}_{(t,e)}^1+Z_{(t,e)}^2\big)\nu(d e) \\
&\qquad + b_{\tilde{z}}(t)\int_{\mathcal{E}}\big(\dot{\tilde{Z}}_{(t,e)}^1+\tilde{Z}_{(t,e)}^2\big)\nu(d e)+\frac{1}{2}\Xi_tD^2b(t)\Xi_t^\top
 +\delta b\left(t,\Delta^1,\Delta^2\right)I_{[\bar{t},\bar{t}+\epsilon]}(t)\bigg\}d t \\
&\quad +\bigg\{\sigma_x(t)\left(X_t^1+X_t^2\right)+\sigma_y(t)\left(Y_t^1+Y_t^2\right)+\sigma_z(t)\int_{\mathcal{E}}\big(\dot{Z}_{(t,e)}^1+Z_{(t,e)}^2\big)\nu(d e) \\
&\qquad +\sigma_{\tilde{z}}(t)\int_{\mathcal{E}}\big(\dot{\tilde{Z}}_{(t,e)}^1+\tilde{Z}_{(t,e)}^2\big)\nu(d e)+\delta\sigma_x\left(t,\Delta^1,\Delta^2\right)X_t^1I_{[\bar{t},\bar{t}+\epsilon]}(t) \\
&\qquad +\delta\sigma_y\left(t,\Delta^1,\Delta^2\right)Y_t^1I_{[\bar{t},\bar{t}+\epsilon]}(t)+\delta\sigma_z\left(t,\Delta^1,\Delta^2\right)\int_{\mathcal{E}}\dot{Z}_{(t,e)}^1\nu(d e)I_{[\bar{t},\bar{t}+\epsilon]}(t) \\
&\qquad +\delta\sigma_{\tilde{z}}\left(t,\Delta^1,\Delta^2\right)\int_{\mathcal{E}}\dot{\tilde{Z}}_{(t,e)}^1\nu(d e)I_{[\bar{t},\bar{t}+\epsilon]}(t)+\frac{1}{2}\Xi_tD^2\sigma(t)\Xi_t^\top +\delta \sigma\left(t,\Delta^1,\Delta^2\right)I_{[\bar{t},\bar{t}+\epsilon]}(t)\bigg\}d W_t \\
&\quad +\int_{\mathcal{E}}\bigg\{f_x(t,e)\left(X_t^1+X_t^2\right)+f_y(t,e)\left(Y_t^1+Y_t^2\right)+f_{\tilde{z}}(t,e)\left(\dot{\tilde{Z}}_{(t,e)}^1+\tilde{Z}_{(t,e)}^2\right)\\
&\qquad +\frac{1}{2}\Xi_tD^2f(t,e)\Xi_t^\top\bigg\}\tilde{N}(d e,d t), \quad X_0^1+X_0^2=0,
\end{aligned}\right.
\end{equation}
and
\begin{equation}\label{variational equation of Y}
\left\{\begin{aligned}
-&d\left(Y_t^1+Y_t^2\right)=\bigg\{g_x(t)\left(X_t^1+X_t^2\right)+g_y(t)\left(Y_t^1+Y_t^2\right)+g_z(t)\int_{\mathcal{E}}\big(\dot{Z}_{(t,e)}^1+Z_{(t,e)}^2\big)\nu(d e) \\
&\quad +g_{\tilde{z}}(t)\int_{\mathcal{E}}\big(\dot{\tilde{Z}}_{(t,e)}^1+\tilde{Z}_{(t,e)}^2\big)\nu(d e)+\frac{1}{2}\Xi_tD^2g(t)\Xi_t^\top +\delta g\left(t,\Delta^1,\Delta^2\right)I_{[\bar{t},\bar{t}+\epsilon]}(t)\bigg\}d t \\
&\quad -\int_{\mathcal{E}}\big(Z_{(t,e)}^1+Z_{(t,e)}^2\big)\nu(d e)d W_t-\int_{\mathcal{E}}\big(\tilde{Z}_{(t,e)}^1+\tilde{Z}_{(t,e)}^2\big)\tilde{N}(d e,d t), \\
& Y_T^1+Y_T^2=\phi_x(\bar{X}_T)\left(X_T^1+X_T^2\right)+\frac{1}{2}\phi_{x x}(\bar{X}_T)\left(X_T^1\right)^2.
\end{aligned}\right.
\end{equation}

Then, we derive the first-order and second-order variational equations from \eqref{variational equation of X} and \eqref{variational equation of Y}. Firstly, it is easy to establish the first-order variational equation for $X_t^1$:
\begin{equation}\label{variational equation of X1}
\left\{\begin{aligned}
dX_t^1&=\bigg\{b_x(t)X_t^1+b_y(t)Y_t^1+b_z(t)\left(\int_{\mathcal{E}}Z_{(t,e)}^1\nu(d e)-\Delta^1(t)I_{[\bar{t},\bar{t}+\epsilon]}(t)\right)\\
&\qquad +b_{\tilde{z}}(t)\left(\int_{\mathcal{E}}\tilde{Z}_{(t,e)}^1\nu(d e)-\Delta^2(t)I_{\mathcal{O}}(t)\right)\bigg\}d t\\
&\quad +\bigg\{\sigma_x(t)X_t^1+\sigma_y(t)Y_t^1+\sigma_z(t)\left(\int_{\mathcal{E}}Z_{(t,e)}^1\nu(d e)-\Delta^1(t)I_{[\bar{t},\bar{t}+\epsilon]}(t)\right) \\
&\qquad\quad +\sigma_{\tilde{z}}(t)\left(\int_{\mathcal{E}}\tilde{Z}_{(t,e)}^1\nu(d e)-\Delta^2(t)I_{\mathcal{O}}(t)\right)+\delta\sigma\left(t,\Delta^1,\Delta^2\right)I_{[\bar{t},\bar{t}+\epsilon]}(t)\bigg\}d W_t \\
&\quad +\int_{\mathcal{E}}\bigg\{f_x(t,e)X_t^1+f_y(t,e)Y_t^1+f_{\tilde{z}}(t,e)\left(\tilde{Z}_{(t,e)}^1-\Delta^2(t)I_{\mathcal{O}}(t)\right)\bigg\}\tilde{N}(d e,d t), \\
X_0^1&=0.
\end{aligned}\right.
\end{equation}

Noting that $Y_T^1=\phi_x(\bar{X}_T)X_T^1$. So we guess that $Y_t^1=p_tX_t^1$, where $\left(p_\cdot,q_{(\cdot,\cdot)},\tilde{q}_{(\cdot,\cdot)}\right)$ satisfies the following adjoint equation:
$$
\left\{\begin{aligned}
-dp_t&=\mathcal{P}(t)d t-\int_{\mathcal{E}}q_{(t,e)} \nu(d e)d W_t-\int_{\mathcal{E}}\tilde{q}_{(t,e)}\tilde{N}(d e,d t),\\
 p_T&=\phi_x(\bar{X}_T),
\end{aligned}\right.
$$
where $\mathcal{P}(\cdot)$ is an $\mathcal{F}_t$-adapted process to be determined.

Applying It\^{o}'s formula for $p_tX_t^1$, we can get
\begin{equation}\label{variational equation of Y1}
\left\{\begin{aligned}
-dY_t^1&=\bigg\{g_x(t)X_t^1+g_y(t)Y_t^1+g_z(t)\left(\int_{\mathcal{E}}Z_{(t,e)}^1\nu(d e)-\Delta^1(t)I_{[\bar{t},\bar{t}+\epsilon]}(t)\right)\\
&\qquad +g_{\tilde{z}}(t)\left(\int_{\mathcal{E}}\tilde{Z}_{(t,e)}^1\nu(d e)-\Delta^2(t)I_{\mathcal{O}}(t)\right)
 -\int_{\mathcal{E}}q_{(t,e)}\nu(d e)\delta\sigma\left(t,\Delta^1,\Delta^2\right)I_{[\bar{t},\bar{t}+\epsilon]}(t)\bigg\}d t \\
&\quad -\int_{\mathcal{E}}Z_{(t,e)}^1 \nu(d e)d W_t-\int_{\mathcal{E}}\tilde{Z}_{(t,e)}^1\tilde{N}(d e,d t),\\
Y_T^1&= \phi_x(\bar{X}_T)X_T^1,
\end{aligned}\right.
\end{equation}
and $\left(p_\cdot,q_{(\cdot,\cdot)},\tilde{q}_{(\cdot,\cdot)}\right)$ is the solution to the following first-order adjoint equation:
\begin{equation}\label{first-order adjoint equation}
\left\{\begin{aligned}
-dp_t&=\bigg\{g_x(t)+g_y(t)p_t+b_x(t)p_t+b_y(t)(p_t)^2+\int_{\mathcal{E}}\Big(g_z(t)K_1(t,e)+g_{\tilde{z}}(t)K_2(t,e) \\
&\qquad +b_z(t)K_1(t,e)p_t+b_{\tilde{z}}(t)K_2(t,e)p_t+\sigma_x(t)q_{(t,e)}+\sigma_y(t)p_tq_{(t,e)}+\sigma_z(t)K_1(t,e)q_{(t,e)} \\
&\qquad +\sigma_{\tilde{z}}(t)K_2(t,e)q_{(t,e)}+f_x(t,e)\tilde{q}_{(t,e)}+f_y(t,e)p_t\tilde{q}_{(t,e)}+f_{\tilde{z}}(t,e)K_2(t,e)\tilde{q}_{(t,e)}\Big)\nu(d e)\bigg\}d t \\
&\quad -\int_{\mathcal{E}}q_{(t,e)} \nu(d e)d W_t-\int_{\mathcal{E}}\tilde{q}_{(t,e)}\tilde{N}(d e,d t),\\
 p_T&=\phi_x(\bar{X}_T).
\end{aligned}\right.
\end{equation}

From the deduction above, it is easy to obtain the following relationships:
\begin{equation} \label{relationship of K1K2}
Y_t^1=p_tX_t^1,\quad Z_{(t,e)}^1=K_1(t,e)X_t^1+p_t\delta\sigma\left(t,\Delta^1\right)I_{[\bar{t},\bar{t}+\epsilon]}(t),\quad \tilde{Z}_{(t,e)}^1=K_2(t,e)X_t^1,
\end{equation}
where $K_1(t,e),K_2(t,e)$ satifies the following equations:
\begin{equation}
\left\{\begin{aligned}
&\sigma_x(t)p_t+\sigma_y(t)(p_t)^2+\int_{\mathcal{E}}\Big(\sigma_z(t)p_tK_1(t,e) +\sigma_{\tilde{z}}(t)p_tK_2(t,e)+q_{(t,e)}\Big)\nu(d e)\\
&=\int_{\mathcal{E}}K_1(t,e)\nu(d e), \\
&f_x(t,e)p_t+f_y(t,e)(p_t)^2+f_{\tilde{z}}(t,e)p_tK_2(t,e)+f_x(t,e)\tilde{q}_{(t,e)}+f_y(t,e)p_t\tilde{q}_{(t,e)} \\
&+f_{\tilde{z}}(t,e)\tilde{q}_{(t,e)}K_2(t,e)+\tilde{q}_{(t,e)}=K_2(t,e),
\end{aligned}\right.
\end{equation}
which implies that
\begin{equation}\label{explicit solution of K1K2}
\left\{\begin{aligned}
\int_{\mathcal{E}}K_1(t,e)\nu(d e)&=\frac{\sigma_x(t)p_t+\sigma_y(t)(p_t)^2+\int_{\mathcal{E}}(\sigma_{\tilde{z}}(t)p_tK_2(t,e)+q_{(t,e)})\nu(d e)}{1-\sigma_z(t)p_t}, \\
K_2(t,e)&=\frac{f_x(t,e)p_t+f_y(t,e)(p_t)^2+f_x(t,e)\tilde{q}_{(t,e)}+f_y(t,e)p_t\tilde{q}_{(t,e)}+\tilde{q}_{(t,e)}}{1-f_{\tilde{z}}(t,e)\left(p_t+\tilde{q}_{(t,e)}\right)}.
\end{aligned}\right.
\end{equation}

From \eqref{definition of Delta12}, \eqref{relationship of K1K2}, we obtain
\begin{equation}\label{representation of Z and Delta}
\dot{Z}_{(t,e)}^1=K_1(t,e)X_t^1,\quad \dot{\tilde{Z}}_{(t,e)}^1=K_2(t,e)X_t^1, \quad
\Delta^1(t)=p_t\delta\sigma\left(t,\Delta^1\right),\quad \Delta^2(t)=0.
\end{equation}

Form \eqref{variational equation of X}, \eqref{variational equation of Y}, \eqref{variational equation of X1}, \eqref{variational equation of Y1} and \eqref{representation of Z and Delta}, it is easy to deduce that $\left(X_\cdot^2,Y_\cdot^2\right)$ satisfies the following equation:
\begin{equation}\label{variational equation of X2Y2}
\left\{\begin{aligned}
dX_t^2&=\bigg\{b_x(t)X_t^2+b_y(t)Y_t^2+\int_{\mathcal{E}}\Big(b_z(t)Z_{(t,e)}^2+b_{\tilde{z}}(t)\tilde{Z}_{(t,e)}^2\Big)\nu(d e)+\delta b\left(t,\Delta^1\right)I_{[\bar{t},\bar{t}+\epsilon]}(t) \\
&\qquad +\frac{1}{2}\tilde{\Xi}_t D^2b(t)\tilde{\Xi}_t^\top\bigg\}d t+\bigg\{\sigma_x(t)X_t^2+\sigma_y(t)Y_t^2+\int_{\mathcal{E}}\Big(\sigma_z(t)Z_{(t,e)}^2+\sigma_{\tilde{z}}(t)\tilde{Z}_{(t,e)}^2\Big)\nu(d e) \\
&\qquad +D\delta\sigma\left(t,\Delta^1\right) \tilde{\Xi}_t^\top I_{[\bar{t},\bar{t}+\epsilon]}(t)+\frac{1}{2}\tilde{\Xi}_tD^2\sigma(t)\tilde{\Xi}_t^\top\bigg\}d W_t \\
&\quad +\int_{\mathcal{E}}\bigg\{f_x(t,e)X_t^2+f_y(t,e)Y_t^2+f_{\tilde{z}}(t,e)\tilde{Z}_{(t,e)}^2+\frac{1}{2}\tilde{\Xi}_tD^2f(t,e)\tilde{\Xi}_t^\top \bigg\}\tilde{N}(d e,d t), \\
-dY_t^2&=\bigg\{g_x(t)X_t^2+g_y(t)Y_t^2+\int_{\mathcal{E}}\big(g_z(t)Z_{(t,e)}^2+g_{\tilde{z}}(t)\tilde{Z}_{(t,e)}^2\big)\nu(d e)+\frac{1}{2}\tilde{\Xi}_tD^2g(t)\tilde{\Xi}_t^\top \\
&\qquad +\left(\int_{\mathcal{E}}q_{(t,e)}\nu(d e)\delta\sigma\left(t,\Delta^1\right)+\delta g\left(t,\Delta^1\right)\right)I_{[\bar{t},\bar{t}+\epsilon]}(t)\bigg\}d t\\
&\quad -\int_{\mathcal{E}}Z_{(t,e)}^2\nu(d e)d W_t-\int_{\mathcal{E}}\tilde{Z}_{(t,e)}^2\tilde{N}(d e,d t), \\
X_0^2&=0,\quad Y_T^2=\phi_x(\bar{X}_T)X_T^2+\frac{1}{2}\phi_{x x}(\bar{X}_T)\left(X_T^1\right)^2,
\end{aligned}\right.
\end{equation}
where $\tilde{\Xi}_t\triangleq\Big[X_t^1,Y_t^1,\int_{\mathcal{E}}K_1(t,e)\nu(d e)X_t^1,\int_{\mathcal{E}}K_2(t,e)\nu(d e)X_t^1\Big]$.

\begin{myrem}
The relation \eqref{representation of Z and Delta} is similar as Remark 3.1 in \cite{ZhengShi2023}. In our paper, motivated by the new variation technique in \cite{SongTangWu2020}, the jump term does not include the term $\delta f(t,\Delta^1,\Delta^2,e)I_{\mathcal{O}}(t)$ and $D\delta f(t,\Delta^1,,\Delta^2,e)\tilde{\Xi}_tI_{\mathcal{O}}(t)$, so $\Delta^2(t)=0$. In addition, the term $f_x(t,e)\tilde{q}_{(t,e)}$ in \eqref{first-order adjoint equation} should be written as $\mathbf{E}\big[f_x(t,e)|\mathscr{P}\otimes\mathscr{B}(\mathcal{E})\big]\tilde{q}_{(t,e)}$ and similar for other terms related to $f$ in the drift term, also in \eqref{explicit solution of K1K2}. Here $\mathbf{E}$ is not an expectation, but it owns similar properties to expectation. In fact, $\mathbf{E}$ is a Radon-Nikodym derivative (see \cite{SongTangWu2020} and \cite{ZhengShi2023} for more details). For similar cases in the subsequent text, we will not be reiterated.
\end{myrem}
Now, we need the following assumption.

\noindent{\bf (H2)} Suppose that the first-order adjoint equation \eqref{first-order adjoint equation} admit a unique bounded solution $\left(p_\cdot,q_{(\cdot,\cdot)},\tilde{q}_{(\cdot,\cdot)}\right)$.

\begin{mylem}\label{high order estimates}
Let {\bf (H1)-(H2)} hold, then for $\beta\geqslant 2$, we have the following estimates:
\begin{equation}\label{order of X1}
\begin{aligned}
&\mathbb{E}\Bigg\{\sup\limits_{0\leqslant t\leqslant T}\left[\big|X_t^1\big|^\beta+\big|Y_t^1\big|^\beta\right]+\left(\int_0^T \big\|Z_{(t,e)}^1\big\|^2 d t\right)^{\frac{\beta}{2}}
+\left(\int_0^T\int_{\mathcal{E}}\big|\tilde{Z}_{(t,e)}^1\big|^2 N(d e,d t)\right)^{\frac{\beta}{2}}\Bigg\}=O\left(\epsilon^{\frac{\beta}{2}}\right),
\end{aligned}
\end{equation}
\begin{equation}\label{order of X-X-X2}
\begin{aligned}
& \mathbb{E}\Bigg\{\sup\limits_{0\leqslant t\leqslant T}\left[\big|X_t^\epsilon-\bar{X}_t-X_t^1\big|^2+\big|Y_t^\epsilon-\bar{Y}_t-Y_t^1\big|^2\right]
+\int_0^T \big\|Z_{(t,e)}^\epsilon-\bar{Z}_{(t,e)}-Z_{(t,e)}^1\big\|^2 d t \\
&\quad +\int_0^T\int_{\mathcal{E}}\big|\tilde{Z}_{(t,e)}^\epsilon-\bar{\tilde{Z}}_{(t,e)}-\tilde{Z}_{(t,e)}^1\big|^2 N(d e,d t)\Bigg\}=O\left(\epsilon^2\right),
\end{aligned}
\end{equation}
\begin{equation}
\begin{aligned}
&\mathbb{E}\Bigg\{\sup\limits_{0\leqslant t\leqslant T}\left[\big|X_t^\epsilon-\bar{X}_t-X_t^1\big|^4+\big|Y_t^\epsilon-\bar{Y}_t-Y_t^1\big|^4\right]
+\left(\int_0^T \big\|Z_{(t,e)}^\epsilon-\bar{Z}_{(t,e)}-Z_{(t,e)}^1\big\|^2 d t\right)^2 \\
&\quad +\left( \int_0^T\int_{\mathcal{E}}\big|\tilde{Z}_{(t,e)}^\epsilon-\bar{\tilde{Z}}_{(t,e)}-\tilde{Z}_{(t,e)}^1\big|^2 N(d e,d t)\right)^2\Bigg\}=o\left(\epsilon^2\right).
\end{aligned}
\end{equation}
\end{mylem}

The proof is in the Appendix.

Next, we give estimates about \eqref{variational equation of X2Y2} and its proof is also in the Appendix.
\begin{mylem}\label{high order estimates of X2 etc}
Let {\bf (H1)-(H2)} hold, then for $\beta\geqslant 2$, we have the following estimates:
\begin{equation}\label{order of X 2}
\begin{aligned}
&\mathbb{E}\bigg\{\sup\limits_{0\leqslant t\leqslant T}\Big[\big|X_t^2\big|^2+\big|Y_t^2\big|^2\Big]
 +\int_0^T \big\|Z_{(t,e)}^2\big\|^2 d t+\int_0^T\int_{\mathcal{E}}\big|\tilde{Z}_{(t,e)}^2\big|^2 N(d e,d t)\bigg\}=O\left(\epsilon^2\right),
\end{aligned}
\end{equation}
\begin{equation}\label{order of X2}
\begin{aligned}
&\mathbb{E}\bigg\{\sup\limits_{0\leqslant t\leqslant T}\left[\big|X_t^2\big|^\beta+\big|Y_t^2\big|^\beta\right]
+\left(\int_0^T \big\|Z_{(t,e)}^2\big\|^2 d t\right)^{\frac{\beta}{2}} +\left(\int_0^T\int_{\mathcal{E}}\big|\tilde{Z}_{(t,e)}^2\big|^2 N(d e,d t)\right)^{\frac{\beta}{2}}\bigg\}=o\left(\epsilon^{\frac{\beta}{2}}\right),
\end{aligned}
\end{equation}
\begin{equation}
Y_0^\epsilon-\bar{Y}_0-Y_0^1-Y_0^2=o(\epsilon).
\end{equation}
\end{mylem}

Note that $Y_t^1=p_tX_t^1$, so $Y_0^1=0$. By Lemma \ref{high order estimates of X2 etc}, we have
$$
J(u_\cdot^\epsilon)-J(\bar{u}_\cdot)=Y_0^\epsilon-\bar{Y}_0=Y_0^2+o(\epsilon).
$$
In order to get $Y_0^2$, we introduce the following second-order adjoint equation:
\begin{equation}\label{second-order adjoint equation}
\left\{\begin{aligned}
-d P_t&=\int_{\mathcal{E}}\bigg\{P_t\Big[\big(D f(t,e)^\top \Xi_2^\top\big)^2+\big(D\sigma(t)^\top \Xi_1^\top\big)^2+2D b(t)^\top \Xi_1^\top+\tilde{q}_{(t,e)}f_y(t,e)\Big]+H_y(t) \\
&\qquad\quad +2Q_{(t,e)}D\sigma(t)^\top \Xi_1^\top+\Xi_1D^2H(t)\Xi_1^\top+\tilde{q}_{(t,e)}\Xi_2D^2f(t,e)\Xi_2^\top \\
&\qquad\quad +2\tilde{Q}_{(t,e)}D f(t,e)^\top\Xi_2^\top+H_z(t)\tilde{K}_1(t,e) +\tilde{Q}_{(t,e)}\big(D f(t,e)^\top\Xi_2^\top\big)^2+H_{\tilde{z}}(t)\tilde{K}_2(t,e) \\
&\qquad\quad +\tilde{q}_{(t,e)}f_{\tilde{z}}(t,e)\tilde{K}_2(t,e)\bigg\}\nu(d e)d t -\int_{\mathcal{E}}Q_{(t,e)}\nu(d e)d W_t-\int_{\mathcal{E}}\tilde{Q}_{(t,e)}\tilde{N}(d e,d t),\\
 P_T&=\phi_{x x}(\bar{X}_T),
\end{aligned}\right.
\end{equation}
where
\begin{equation}
\begin{aligned}
&\quad H(t,x,y,z,\tilde{z},u,p,q)\triangleq g(t,x,y,z,\tilde{z},u)+p b(t,x,y,z,\tilde{z},u)+q\sigma(t,x,y,z,\tilde{z},u), \\
&\tilde{K}_1(t,e)\triangleq R_1\Big[p_t\sigma_y(t)P_t+2D\sigma(t)^\top\Xi_1^\top P_t+Q_{(t,e)}+p_t\Xi_1D^2\sigma(t)\Xi_1^\top+\sigma_{\tilde{z}}(t)p_t\tilde{K}_2(t,e)\Big], \\
&\tilde{K}_2(t,e)\triangleq R_2\Big[f_y(t,e)mP_t+m\Xi_2D^2f(t,e)\Xi_2^\top+2n D f(t,e)^\top\Xi_2^\top+n\big(D f(t,e)^\top\Xi_2^\top\big)^2+\tilde{Q}_{(t,e)}\Big], \\
&R_1\triangleq \big(1-\sigma_z(t)p_t\big)^{-1}, \quad R_2\triangleq \big(1-f_{\tilde{z}}(t,e)m\big)^{-1},\quad m\triangleq p_t+\tilde{q}_{(t,e)},\quad n\triangleq P_t+\tilde{Q}_{(t,e)}, \\
&\qquad \Xi_1\triangleq \left[1,p_t,\int_{\mathcal{E}}K_1(t,e)\nu(d e),\int_{\mathcal{E}}K_2(t,e)\nu(d e)\right],\quad \Xi_2\triangleq \left[1,p_t,\int_{\mathcal{E}}K_2(t,e)\nu(d e)\right].
\end{aligned}
\end{equation}
Note that (\ref{second-order adjoint equation}) is a linear BSDEP with uniformly Lipschitz continuous coefficients, then it has a unique solution $\big(P_\cdot,Q_{(\cdot,\cdot)},\tilde{Q}_{(\cdot,\cdot)}\big)$. Next, we introduce the following equation:
\begin{equation}\label{Y* equation}
\left\{\begin{aligned}
-dY_t^*&=\int_{\mathcal{E}}\bigg\{\Big(H_y(t)+\tilde{q}_{(t,e)}f_y(t,e)+H_z(t)\sigma_y(t)p_tR_1+H_{\tilde{z}}(t)f_y(t,e)mR_2\Big)Y_t^* \\
&\qquad +\Big(H_z(t)+H_z(t)\sigma_z(t)p_tR_1\Big)Z_{(t,e)}^*+\Big[\delta H(t,\Delta^1)+\frac{1}{2}P_t\big(\delta\sigma(t,\Delta^1)\big)^2\Big]I_{[\bar{t},\bar{t}+\epsilon]}(t) \\
&\qquad +\Big(H_{\tilde{z}}(t)+\tilde{q}_{(t,e)}f_{\tilde{z}}(t,e)+H_z(t)\sigma_{\tilde{z}}(t)p_tR_1\big(1+R_2f_{\tilde{z}}(t,e)m\big) \\
&\qquad +H_{\tilde{z}}(t)f_{\tilde{z}}(t,e)mR_2\Big)\tilde{Z}_{(t,e)}^*\bigg\}\nu(d e)d t-\int_{\mathcal{E}}Z_{(t,e)}^*\nu(d e)d W_t-\int_{\mathcal{E}}\tilde{Z}_{(t,e)}^*\tilde{N}(d e,d t), \\
 Y_T&=0,
\end{aligned}\right.
\end{equation}
where
$$
\delta H(t,\Delta^1)\triangleq \delta g(t,\Delta^1)+p_t \delta b(t,\Delta^1)+q_{(t,e)}\delta \sigma(t,\Delta^1).
$$

\begin{mylem}\label{relationships}
Let {\bf (H1)-(H2)} hold, then we have
\begin{equation}
\left\{\begin{aligned}
Y_t^2&=p_t X_t^2+\frac{1}{2}P_t(X_t^1)^2+Y_t^*,\\
Z_{(t,e)}^2&=\mathbf{I}_1(t,e)+Z_{(t,e)}^*, \\
\tilde{Z}_{(t,e)}^2&=\mathbf{I}_2(t,e)+\tilde{Z}_{(t,e)}^*,
\end{aligned}\right.
\end{equation}
where $\big(Y_\cdot^*,Z_{(\cdot,\cdot)}^*,\tilde{Z}_{(\cdot,\cdot)}^*\big)$ is the solution to \eqref{Y* equation} and
$$
\begin{aligned}
\mathbf{I}_1(t,e)&\triangleq K_1(t,e)X_t^2+\frac{1}{2} \tilde{K}_1(t,e) (X_t^1)^2+R_1\Big[\sigma_y(t)p_t Y_t^*+\sigma_z(t)p_t Z_t^*+\sigma_{\tilde{z}}(t)p_t \tilde{Z}_t^*\Big] \\
&\quad +R_1R_2\sigma_{\tilde{z}}(t)p_t \Big[f_y(t,e)mY_t^*+f_{\tilde{z}}(t,e)m\tilde{Z}_{(t,e)}^*\Big]\\
&\quad +R_1\Big[P_t\delta\sigma(t,\Delta^1)+p_t D\delta\sigma(t,\Delta^1)\Xi_1 \Big]X_t^1I_{[\bar{t},\bar{t}+\epsilon]}(t), \\
\mathbf{I}_2(t,e)&\triangleq K_2(t,e)X_t^2+\frac{1}{2} \tilde{K}_2(t,e) (X_t^1)^2+R_2f_y(t,e)mY_t^* +R_2f_{\tilde{z}}(t,e)m\tilde{Z}_{(t,e)}^*.\\
\end{aligned}
$$
\end{mylem}

\begin{proof}
Applying It\^{o}'s formula for $p_tX_t^2$ and $P_t(X_t^1)^2$, and we can obtain the above ralationships. It is worth noting that there are several terms that can not be reconciled here, as we do not give explicit formula for $L(t)$ in second-order variational equation \eqref{variational equation of X2Y2} by $L(t)X_t^1I_{[\bar{t},\bar{t}+\epsilon]}(t)\sim o(\epsilon)$ where $L(t)$ can be determined by  It\^{o}'s formula.
\end{proof}

Finally, consider the following forward equation:
\begin{equation}
\left\{\begin{aligned}
d\gamma(t)&= \int_{\mathcal{E}}\bigg\{\gamma(t)\Big[H_y(t)+\tilde{q}_{(t,e)}f_y(t,e)+H_z(t)\sigma_y(t)p_tR_1+H_{\tilde{z}}(t)f_y(t,e)mR_2\Big]\bigg\}\nu(d e)d t \\
&\quad +\int_{\mathcal{E}}\bigg\{\gamma(t)\Big[H_z(t)+H_z(t)\sigma_z(t)p_tR_1\Big]\bigg\}\nu(d e)d W_t+\int_{\mathcal{E}}\bigg\{\gamma(t)\Big[H_{\tilde{z}}(t)+\tilde{q}_{(t,e)}f_{\tilde{z}}(t,e)\\
&\quad +H_z(t)\sigma_{\tilde{z}}(t)p_tR_1\big(1+R_2f_{\tilde{z}}(t,e)m\big)+H_{\tilde{z}}(t)f_{\tilde{z}}(t,e)mR_2\Big]\bigg\}\tilde{N}(d e,d t), \\
\gamma(0)&=1.
\end{aligned}\right.
\end{equation}
Applying It\^{o}'s formula for $\gamma(t)Y_t^*$, we have
\begin{equation}
\begin{aligned}
Y_0^*=\mathbb{E}\int_0^T\gamma(t)\left[\delta H(t,\Delta^1)+\frac{1}{2}P_t\Big(\delta\sigma(t,\Delta^1)\Big)^2\right]I_{[\bar{t},\bar{t}+\epsilon]}(t) d t.
\end{aligned}
\end{equation}

Define
\begin{equation}
\begin{aligned}
&\quad \mathcal{H}(t,x,y,z,\tilde{z},u,p,q,P)\\
&\triangleq pb(t,x,y,z+\Delta(t),\tilde{z},u)+q\sigma(t,x,y,z+\Delta(t),\tilde{z},u)+g(t,x,y,z+\Delta(t),\tilde{z},u) \\
&\quad +\frac{1}{2}P\left(\sigma(t,x,y,z+\Delta(t),\tilde{z},u)-\sigma(t,\bar{X}_t,\bar{Y}_t,\bar{Z}_{(t,e)},\bar{\tilde{Z}}_{(t,e)},\bar{u}_t)\right)^2,
\end{aligned}
\end{equation}
where $\Delta(t)$ is as follows by \eqref{representation of Z and Delta}:
$$
\Delta(t)=\Delta^1(t)=p_t\left(\sigma\big(t,\bar{X}_t,\bar{Y}_t,\bar{Z}_{(t,e)}+\Delta(t),\bar{\tilde{Z}}_{(t,e)},{u}\big)-\sigma\big(t,\bar{X}_t,\bar{Y}_t,\bar{Z}_{(t,e)},\bar{\tilde{Z}}_{(t,e)},\bar{u}_t\big)\right).
$$
It is easy to check that
$$
\begin{aligned}
&\ \delta H(t,\Delta^1)+\frac{1}{2}P_t\big(\delta\sigma(t,\Delta^1)\big)^2 \\
=&\ \mathcal{H}\big(t,\bar{X}_t,\bar{Y}_t,\bar{Z}_{(t,e)},\bar{\tilde{Z}}_{(t,e)},u,p_t,q_{(t,e)},P_t\big)-\mathcal{H}\big(t,\bar{X}_t,\bar{Y}_t,\bar{Z}_{(t,e)},\bar{\tilde{Z}}_{(t,e)},\bar{u}_t,p_t,q_{(t,e)},P_t\big).
\end{aligned}
$$
Noting that $\gamma(t)>0$ for $t\in[0,T]$, then we obtain the main result of this paper.

\begin{mythm}\label{General maximum principle}
{\bf (General maximum principle)}
Let {\bf (H1)-(H2)} hold, Suppose that $\bar{u}_\cdot$ is an optimal control for {\bf Problem (SROCPJ)}, and $\big(\bar{X}_\cdot, \bar{Y}_\cdot, \bar{Z}_\cdot,\bar{\tilde{Z}}_{(\cdot,\cdot)}\big)\in \mathscr{M}^p[0,T]$ is the corresponding optimal trajectory. Let $\left(p_\cdot,q_{(\cdot,\cdot)},\tilde{q}_{(\cdot,\cdot)}\right)$ and $\big(P_\cdot,Q_{(\cdot,\cdot)},\tilde{Q}_{(\cdot,\cdot)}\big)$  be the solution to adjoint BSDEPs \eqref{first-order adjoint equation}, \eqref{second-order adjoint equation}, respectively. Then the following inequality holds:
\begin{equation}
\begin{aligned}
&\mathcal{H}\big(t,\bar{X}_t,\bar{Y}_t,\bar{Z}_{(t,e)},\bar{\tilde{Z}}_{(t,e)},u,p_t,q_{(t,e)},P_t\big)
\geqslant \mathcal{H}\big(t,\bar{X}_t,\bar{Y}_t,\bar{Z}_{(t,e)},\bar{\tilde{Z}}_{(t,e)},\bar{u}_t,p_t,q_{(t,e)},P_t\big),\\
&\qquad\qquad\qquad\qquad\qquad\qquad\qquad \forall\ u\in \mathbf{U},\ \text{a.e.} \ \text{a.s.}.
\end{aligned}
\end{equation}
\end{mythm}

\begin{myrem}
It is worth noting that the result we obtained can include many classical theories. If $f\equiv 0$, $b,\sigma$ are independent of $(y,z,\tilde{z})$, $g$ is independent of $\tilde{z}$, then Theorem \ref{General maximum principle} reduces to Theorem 2 of \cite{Hu2017}. If $f\equiv 0$, $b,\sigma,g$ are independent of $\tilde{z}$, then Theorem \ref{General maximum principle} reduces to Theorem 3.18 of \cite{HuJiXue2018}. If $b,\sigma,f$ are independent of $y,z,\tilde{z}$, then Theorem \ref{General maximum principle} reduces to Theorem 2.1 of \cite{WangShi2024}.
\end{myrem}

\section{Concluding remarks}

In this paper, we have derived a general maximum principle for fully coupled forward-backward stochastic systems with jumps, where the control domain is not necessarily convex and the control variable and jump size are integrated into all coefficients. The results obtained in this paper cover the case without jumps (\cite{Peng1990}, \cite{Hu2017}, \cite{HuJiXue2018}) and the case with jumps but the state equation is not fully coupled (\cite{WangShi2024}).

As stated in \cite{SongTangWu2020}, our type of general maximum principles only describes the optimal control on the area that $N$ is continuous, it possesses no information about the optimal control on the time when $N$ jumps. How to find a way to characterize the optimal control on the time when $N$ jumps is a rather challenging problem. We will consider this topic in the future.

\section*{Appendix}\setcounter{section}{0}

\subsection{Proof of Proposition 2.1}

\begin{proof}
Consider the decoupled FBSDEP \eqref{decoupled FBSDEPs} and denote an operator by $\Gamma:\left(y_\cdot,z_{(\cdot,\cdot)},\tilde{z}_{(\cdot,\cdot)}\right)\rightarrow\\\big(Y_\cdot,Z_{(\cdot,\cdot)},\tilde{Z}_{(\cdot,\cdot)}\big)$. It is easy to deduce that the solution $\big(Y_\cdot,Z_{(\cdot,\cdot)},\tilde{Z}_{(\cdot,\cdot)}\big)$ of \eqref{decoupled FBSDEPs} belongs to $\mathscr{N}^p[0,T]$ under {\bf (H0)}, so the mapping $\Gamma$ is well defined.

Next, we show that $\Gamma$ is a strict contraction. For two elements $\big(y_\cdot^i,z_{(\cdot,\cdot)}^i,\tilde{z}_{(\cdot,\cdot)}^i\big)\in \mathscr{N}^p[0,T]$, $i=1,2$, let $\big(X_\cdot^i,Y_\cdot^i,Z_{(\cdot,\cdot)}^i,\tilde{Z}_{(\cdot,\cdot)}^i\big)$ be the corresponding solution to \eqref{decoupled FBSDEPs}, $i=1,2$.

Set
$$
\begin{array}{c}
\Delta y\triangleq y^1-y^2,\quad \Delta z\triangleq z^1-z^2,\quad \Delta \tilde{z}\triangleq \tilde{z}^1-\tilde{z}^2,\\[2ex]
\Delta X\triangleq X^1-X^2,\quad \Delta Y\triangleq Y^1-Y^2,\quad \Delta Z\triangleq Z^1-Z^2,\quad \Delta \tilde{Z}\triangleq \tilde{Z}^1-\tilde{Z}^2.
\end{array}
$$
Note that
$$
\begin{aligned}
&\ b\big(t,X_t^1,y_t^1,z_{(t,e)}^1,\tilde{z}_{(t,e)}^1\big)-b\big(t,X_t^2,y_t^2,z_{(t,e)}^2,\tilde{z}_{(t,e)}^2\big)\\
=&\ b\big(t,X_t^1,y_t^1,z_{(t,e)}^1,\tilde{z}_{(t,e)}^1\big)-b\big(t,X_t^2,y_t^1,z_{(t,e)}^1,\tilde{z}_{(t,e)}^1\big)+b\big(t,X_t^2,y_t^1,z_{(t,e)}^1,\tilde{z}_{(t,e)}^1\big)\\
&\ -b\big(t,X_t^2,y_t^2,z_{(t,e)}^1,\tilde{z}_{(t,e)}^1\big)+ b\big(t,X_t^2,y_t^2,z_{(t,e)}^1,\tilde{z}_{(t,e)}^1\big)-b\big(t,X_t^2,y_t^2,z_{(t,e)}^2,\tilde{z}_{(t,e)}^1\big)\\
&\ + b\big(t,X_t^2,y_t^2,z_{(t,e)}^2,\tilde{z}_{(t,e)}^1\big)-b\big(t,X_t^2,y_t^2,z_{(t,e)}^2,\tilde{z}_{(t,e)}^2\big)\\
=&\ b_1(t)\Delta X_t+b_2(t)\Delta y_t+b_3(t)\int_{\mathcal{E}}\Delta z_{(t,e)}\nu(d e)+b_4(t)\int_{\mathcal{E}}\Delta \tilde{z}_{(t,e)}\nu(d e),
\end{aligned}
$$
where
$$
b_1(t)\triangleq
\begin{cases}
\frac{b\big(t,X_t^1,y_t^1,z_{(t,e)}^1,\tilde{z}_{(t,e)}^1\big)-b\big(t,X_t^2,y_t^1,z_{(t,e)}^1,\tilde{z}_{(t,e)}^1\big)}{\Delta X_t}, & \text{if}\ \Delta X_t\neq 0,\\
0, & \text{if}\ \Delta X_t= 0,
\end{cases}
$$
and similarly for $b_i(t),\sigma_i(t),f_i(t,e),g_i(t),\phi_1(t),i=1,2,3,4$. Then
\begin{equation}\label{DeltaXY}
\left\{\begin{aligned}
d\Delta X_t&= \left\{b_1(t)\Delta X_t+b_2(t)\Delta y_t+\int_{\mathcal{E}}\left(b_3(t)\Delta z_{(t,e)}+b_4(t)\Delta \tilde{z}_{(t,e)}\right)\nu(d e)\right\}dt\\
 &\ + \left\{\sigma_1(t)\Delta X_t+\sigma_2(t)\Delta y_t+\int_{\mathcal{E}}\left(\sigma_3(t)\Delta z_{(t,e)}+\sigma_4(t)\Delta \tilde{z}_{(t,e)}\right)\nu(d e)\right\}d W_t\\
&\ +\int_{\mathcal{E}}\left\{f_1(t,e)\Delta X_t+f_2(t,e)\Delta y_t+f_4(t,e)\Delta \tilde{z}_{(t,e)}\right\}\tilde{N}(d e,d t),\\
-d\Delta Y_t&= \left\{g_1(t)\Delta X_t+g_2(t)\Delta Y_t+\int_{\mathcal{E}}\big(g_3(t)\Delta Z_{(t,e)}+g_4(t)\Delta \tilde{Z}_{(t,e)}\big)\nu(d e)\right\} dt\\
&\quad -\int_{\mathcal{E}}\Delta Z_{(t,e)}\nu(d e)d W_t-\int_{\mathcal{E}}\Delta \tilde{Z}_{(t,e)}\tilde{N}(d e,d t),\\
\Delta X_0&=0,\quad \Delta Y_T=\phi_1(T)\Delta X_T.
\end{aligned}\right.
\end{equation}

Furthermore, for $\psi=b,\sigma,f,g,\phi$, $\psi_i(t),i=1,2,3,4$ are bounded by Lipschitz constants of the corresponding coefficients. Due to Lemma \ref{lemma 2.1}, we have
\begin{equation}
\begin{aligned}
&\mathbb{E}\Bigg\{\sup\limits_{0\leqslant t\leqslant T}\left[\big|\Delta X_t\big|^p+\big|\Delta Y_t\big|^p\right]+\left(\int_0^T\big\|\Delta Z_{(t,e)}\big\|^2 d t\right)^{\frac{p}{2}}+\left(\int_0^T\int_{\mathcal{E}}\big|\Delta\tilde{Z}_{(t,e)}\big|^2N(d e,d t)\right)^{\frac{p}{2}}\Bigg\}\\
\leqslant&\ C_p \mathbb{E}\Bigg\{\left(\int_0^T\Big(\big|b_2(t)\big|\big|\Delta y_t\big|+\big|b_3(t)\big|\big\|\Delta z_{(t,e)}\big\|+\big|b_4(t)\big|\big\|\Delta \tilde{z}_{(t,e)}\big\|\Big)d t\right)^p \\
&\qquad +\left(\int_0^T\Big(\big|\sigma_2(t)\big|\big|\Delta y_t\big|+\big|\sigma_3(t)\big|\big\|\Delta z_{(t,e)}\big\|+\big|\sigma_4(t)\big|\big\|\Delta \tilde{z}_{(t,e)}\big\|\Big)^2d t\right)^{\frac{p}{2}}\\
&\qquad +\left(\int_0^T\int_{\mathcal{E}}\Big(\big|f_2(t,e)\big|\big|\Delta y_t\big|+\big|f_4(t,e)\big|\big|\Delta \tilde{z}_{(t,e)}\big|\Big)^2 N(d e,d t)\right)^{\frac{p}{2}}\Bigg\}.
\end{aligned}
\end{equation}
Noting that
$$
\begin{aligned}
\left(\int_0^T\big\|\Delta z_{(t,e)}\big\|d t\right)^p&\leqslant C\left(\int_0^T\big\|\Delta z_{(t,e)}\big\|^2 d t\right)^{\frac{p}{2}},\\
\left(\int_0^T\big\|\Delta \tilde{z}_{(t,e)}\big\|d t\right)^p
&\leqslant C\left(\int_0^T\int_{\mathcal{E}}\big|\Delta \tilde{z}_{(t,e)}\big|^2 \nu(d e)d t\right)^{\frac{p}{2}}\\
&\leqslant C\left(\frac{p}{2}\right)^{\frac{p}{2}} \left(\int_0^T\int_{\mathcal{E}}\big|\Delta \tilde{z}_{(t,e)}\big|^2 N(d e,d t)\right)^{\frac{p}{2}},
\end{aligned}
$$
the last inequality can be obtained from \cite{LiWei2014} or \cite{ZhengShiarXiv2023}. Then we get
\begin{equation}\label{contraction mapping}
\begin{aligned}
&\mathbb{E}\Bigg\{\sup\limits_{0\leqslant t\leqslant T}\left[\big|\Delta X_t\big|^p+\big|\Delta Y_t\big|^p\right]
+\left(\int_0^T\big\|\Delta Z_{(t,e)}\big\|^2 d t\right)^{\frac{p}{2}}+\left(\int_0^T\int_{\mathcal{E}}\big|\Delta\tilde{Z}_{(t,e)}\big|^2N(d e,d t)\right)^{\frac{p}{2}}\Bigg\}\\
&\leqslant 2\left(2p\right)^{\frac{p}{2}}C_p C_0^p(1+T^p)\mathbb{E}\Bigg\{\sup\limits_{0\leqslant t\leqslant T}\big|\Delta y_t\big|^p
+\left(\int_0^T\big\|\Delta z_{(t,e)}\big\|^2 d t\right)^{\frac{p}{2}}\\
&\qquad +\left(\int_0^T\int_{\mathcal{E}}\big|\Delta \tilde{z}_{(t,e)}\big|^2 N(d e,d t)\right)^{\frac{p}{2}}\Bigg\}\\
&= \Lambda_p \mathbb{E}\Bigg\{\sup\limits_{0\leqslant t\leqslant T}\big|\Delta y_t\big|^p+\left(\int_0^T\big\|\Delta z_{(t,e)}\big\|^2 d t\right)^{\frac{p}{2}}
+\left(\int_0^T\int_{\mathcal{E}}\big|\Delta \tilde{z}_{(t,e)}\big|^2 N(d e,d t)\right)^{\frac{p}{2}}\Bigg\}.
\end{aligned}
\end{equation}

Since $\Lambda_p<1$, the operator $\Gamma$ is a contraction mapping and \eqref{fully coupled FBSDEPs} admits a unique solution $\big(X_\cdot,Y_\cdot,Z_{(\cdot,\cdot)},\tilde{Z}_{(\cdot,\cdot)}\big)\in \mathscr{M}^p[0,T]$. Now we let $\big(X_\cdot^0,Y_\cdot^0,Z_{(\cdot,\cdot)}^0,\tilde{Z}_{(\cdot,\cdot)}^0\big)\in \mathscr{M}^p[0,T]$ be the solution to \eqref{decoupled FBSDEPs} with $y=0,z=0,\tilde{z}=0$. From \eqref{contraction mapping},
$$
\big\|\big(Y-Y^0,Z-Z^0,\tilde{Z}-\tilde{Z}^0\big)\big\|_{\mathscr{N}^p}\leqslant \Lambda_p^{1/p}\big\|\big(Y-0,Z-0,\tilde{Z}-0\big)\big\|_{\mathscr{N}^p}\equiv\Lambda_p^{1/p}\big\|\big(Y,Z,\tilde{Z}\big)\big\|_{\mathscr{N}^p}.
$$

By triangle inequality,
$$
\big\|\big(Y,Z,\tilde{Z}\big)\big\|_{\mathscr{N}^p}\leqslant\big\|\big(Y-Y^0,Z-Z^0,\tilde{Z}-\tilde{Z}^0\big)\big\|_{\mathscr{N}^p}+\big\|\big(Y^0,Z^0,\tilde{Z}^0\big)\big\|_{\mathscr{N}^p},
$$
which leads to
$$
\big\|\big(Y,Z,\tilde{Z}\big)\big\|_{\mathscr{N}^p}\leqslant\left(1-\Lambda_p^{1/p}\right)^{-1}\big\|\big(Y^0,Z^0,\tilde{Z}^0\big)\big\|_{\mathscr{N}^p}.
$$

By Lemma \ref{lemma 2.1}, we have
$$
\begin{aligned}
\big\|\big(Y^0,Z^0,\tilde{Z}^0\big)\big\|_{\mathscr{N}^p}^p
&\leqslant C_p\mathbb{E}\Bigg\{\big|\phi(0)\big|^p+\big|x_0\big|^p+\left(\int_0^T \big|b(t,0,0,0,0)\big|d t\right)^p +\left(\int_0^T \big|g(t,0,0,0,0)\big|d t\right)^p \\
&\qquad\quad +\left(\int_0^T\big|\sigma(t,0,0,0,0)\big|^2 d t\right)^{\frac{p}{2}}+\left(\int_0^T \int_{\mathcal{E}}\big|f(t,0,0,0,e)\big|^2 N(d e,d t)\right)^{\frac{p}{2}}\Bigg\}.
\end{aligned}
$$
Thus we obtain
$$
\begin{aligned}
\big\|\big(Y,Z,\tilde{Z}\big)\big\|_{\mathscr{N}^p}^p
&\leqslant C^{\prime}\mathbb{E}\Bigg\{\big|\phi(0)\big|^p+\big|x_0\big|^p+\left(\int_0^T \big|b(t,0,0,0,0)\big|d t\right)^p +\left(\int_0^T \big|g(t,0,0,0,0)\big|d t\right)^p \\
&\qquad\quad +\left(\int_0^T\big|\sigma(t,0,0,0,0)\big|^2 d t\right)^{\frac{p}{2}}+\left(\int_0^T \int_{\mathcal{E}}\big|f(t,0,0,0,e)\big|^2 N(d e,d t)\right)^{\frac{p}{2}}\Bigg\},
\end{aligned}
$$
where $C^{\prime}=C_p\left(1-\Lambda_p^{1/p}\right)^{-p}$. By Lemma \ref{lemma 2.1} again, we get the final result.
\end{proof}

\subsection{Proof of Lemma 3.1}
\begin{proof}
For simplicity, we denote
$$
\begin{array}{c}
\Theta_t\triangleq\big(\bar{X}_t,\bar{Y}_t,\bar{Z}_{(t,e)},\bar{\tilde{Z}}_{(t,e)}\big),\quad \Theta_t^\epsilon\triangleq\big(X_t^\epsilon,Y_t^\epsilon,Z_{(t,e)}^\epsilon,\tilde{Z}_{(t,e)}^\epsilon\big),\\[2ex]
\hat{X}_t\triangleq X_t^\epsilon-\bar{X}_t,\quad \hat{Y}_t\triangleq Y_t^\epsilon-\bar{Y}_t,\quad \hat{Z}_{(t,e)}\triangleq Z_{(t,e)}^\epsilon-\bar{Z}_{(t,e)},\quad \hat{\tilde{Z}}_{(t,e)}\triangleq\tilde{Z}_{(t,e)}^\epsilon-\bar{\tilde{Z}}_{(t,e)}.
\end{array}
$$
Note that
$$
\begin{aligned}
&\ b\left(t,\Theta_t^\epsilon,u_t^\epsilon\right)-b\left(t,{\Theta}_t,\bar{u}_t\right) \\
=&\ b\left(t,\Theta_t^\epsilon,u_t^\epsilon\right)-b\left(t,{\Theta}_t,u_t^\epsilon\right)+b\left(t,{\Theta}_t,u_t^\epsilon\right)-b\left(t,{\Theta}_t,\bar{u}_t\right) \\
=&\ b\big(t,\hat{X}_t+\bar{X}_t,\hat{Y}_t+\bar{Y}_t,\hat{Z}_{(t,e)}+\bar{Z}_{(t,e)},\hat{\tilde{Z}}_{(t,e)}+\bar{\tilde{Z}}_{(t,e)},u_t^\epsilon\big)
 -b\left(t,\Theta_t,u_t^\epsilon\right)+\delta b(t)I_{[\bar{t},\bar{t}+\epsilon]}(t) \\
=&\ \tilde{b}_x(t)\hat{X}_t+\tilde{b}_y(t)\hat{Y}_t+\tilde{b}_z(t)\int_{\mathcal{E}}\hat{Z}_{(t,e)}\nu(d e)+\tilde{b}_{\tilde{z}}(t)\int_{\mathcal{E}}\hat{\tilde{Z}}_{(t,e)}\nu(d e)
 +\delta b(t)I_{[\bar{t},\bar{t}+\epsilon]}(t), \\
\end{aligned}
$$
where
\begin{equation} \label{definition of b}
\tilde{b}_x(t)\triangleq\int_0^1 b_x\big(t,\theta\hat{X}_t+\bar{X}_t,\theta\hat{Y}_t+\bar{Y}_t,\theta\hat{Z}_{(t,e)}+\bar{Z}_{(t,e)},\theta\hat{\tilde{Z}}_{(t,e)}+\bar{\tilde{Z}}_{(t,e)},u_t^\epsilon\big)d \theta,
\end{equation}
and similarly for $b,\sigma,f,g,\phi$ and their derivatives of $x,y,z,\tilde{z}$. Then we get
\begin{equation}
\left\{\begin{aligned}
d\hat{X}_t&=\bigg\{\tilde{b}_x(t)\hat{X}_t+\tilde{b}_y(t)\hat{Y}_t+\int_{\mathcal{E}}\Big(\tilde{b}_z(t)\hat{Z}_{(t,e)}+\tilde{b}_{\tilde{z}}(t)\hat{\tilde{Z}}_{(t,e)}\Big)\nu(d e)
 +\delta b(t)I_{[\bar{t},\bar{t}+\epsilon]}(t)\bigg\}d t \\
&\quad +\bigg\{\tilde{\sigma}_x(t)\hat{X}_t+\tilde{\sigma}_y(t)\hat{Y}_t+\int_{\mathcal{E}}\Big(\tilde{\sigma}_z(t)\hat{Z}_{(t,e)}+\tilde{\sigma}_{\tilde{z}}(t)\hat{\tilde{Z}}_{(t,e)}\Big)\nu(d e)
 +\delta \sigma(t)I_{[\bar{t},\bar{t}+\epsilon]}(t)\bigg\}d W_t \\
&\quad +\int_{\mathcal{E}}\bigg\{\tilde{f}_x(t,e)\hat{X}_t+\tilde{f}_y(t,e)\hat{Y}_t+\tilde{f}_{\tilde{z}}(t,e)\hat{\tilde{Z}}_{(t,e)}+\delta f(t,e)I_{\mathcal{O}}(t)\bigg\}\tilde{N}(d e,d t), \\
-d\hat{Y}_t&=\bigg\{\tilde{g}_x(t)\hat{X}_t+\tilde{g}_y(t)\hat{Y}_t+\int_{\mathcal{E}}\Big(\tilde{g}_z(t)\hat{Z}_{(t,e)}+\tilde{g}_{\tilde{z}}(t)\hat{\tilde{Z}}_{(t,e)}\Big)\nu(d e)+\delta g(t)I_{[\bar{t},\bar{t}+\epsilon]}(t)\bigg\}d t \\
&\quad-\int_{\mathcal{E}}\hat{Z}_{(t,e)} \nu(d e)d W_t-\int_{\mathcal{E}}\hat{\tilde{Z}}_{(t,e)}\tilde{N}(d e,d t), \\
\hat{X}_0&=0, \quad \hat{Y}_T=\ \tilde{\phi}_x(T)\hat{X}_T.
\end{aligned}\right.
 \label{variational equation FBSDEP}
\end{equation}

Noting that $\big(\hat{X}_t,\hat{Y}_t,\hat{Z}_{(t,e)},\hat{\tilde{Z}}_{(t,e)}\big)$ is the solution to \eqref{variational equation FBSDEP}, by H\"{o}lder's inequality, we have
$$
\mathbb{E}\left(\int_{\bar{t}}^{\bar{t}+\epsilon}\big|u_t\big|d t\right)^\beta\leqslant \epsilon^{\beta-1} \mathbb{E}\int_{\bar{t}}^{\bar{t}+\epsilon}\big|u_t\big|^\beta d t
\leqslant\epsilon^{\beta}\mathbb{E}\left[\sup\limits_{0\leqslant t\leqslant T}\big|u_t\big|^\beta\right].
$$
Then, by Proposition \ref{Lp-estimate of fully coupled FBSDEP}, we have
$$
\begin{aligned}
&\mathbb{E}\Bigg\{\sup\limits_{0\leqslant t\leqslant T}\Big[\big|X_t^\epsilon-\bar{X}_t\big|^\beta+\big|Y_t^\epsilon-\bar{Y}_t\big|^\beta\Big] \\
&\quad +\left(\int_0^T \big\|Z_{(t,e)}^\epsilon-\bar{Z}_{(t,e)}\big\|^2 d t\right)^{\frac{\beta}{2}}
 +\left(\int_0^T\int_{\mathcal{E}}\big|\tilde{Z}_{(t,e)}^\epsilon-\bar{\tilde{Z}}_{(t,e)}\big|^2 N(d e,d t)\right)^{\frac{\beta}{2}}\Bigg\} \\
\leqslant &\ C\mathbb{E}\Bigg\{\left(\int_0^T\big|\delta b(t)\big|I_{[\bar{t},\bar{t}+\epsilon]}(t)d t\right)^\beta+\left(\int_0^T\big|\delta g(t)\big|I_{[\bar{t},\bar{t}+\epsilon]}(t)d t\right)^\beta \\
&\qquad +\left(\int_0^T \big|\delta \sigma(t)\big|^2I_{[\bar{t},\bar{t}+\epsilon]}(t)d t\right)^{\frac{\beta}{2}}
 +\left(\int_0^T\int_{\mathcal{E}} \big|\delta f(t,e)\big|^2I_{\mathcal{O}}(t)N(d e,d t)\right)^{\frac{\beta}{2}}\Bigg\} \\
\leqslant &\ C\mathbb{E}\Bigg\{\left(\int_{\bar{t}}^{\bar{t}+\epsilon}\left(1+\big|\bar{X}_t\big|+\big|\bar{Y}_t\big|+\big|u_t\big|+\big|\bar{u}_t\big|\right)d t\right)^\beta \\
&\qquad +\left(\int_{\bar{t}}^{\bar{t}+\epsilon}\left(1+\big|\bar{X}_t\big|+\big|\bar{Y}_t\big|+\big|u_t\big|+\big|\bar{u}_t\big|\right)^2d t\right)^{\frac{\beta}{2}}
 +\left(\int_0^T\int_{\mathcal{E}} \big|\delta f(t,e)\big|^2I_{\mathcal{O}}(t)N(d e,d t)\right)^{\frac{\beta}{2}}\Bigg\} \\
\leqslant &\ C\left(\epsilon^\beta+\epsilon^{\frac{\beta}{2}}\right)\mathbb{E}
\bigg\{1+\sup\limits_{0\leqslant t\leqslant T}\left[\big|\bar{X}_t\big|^\beta+\big|\bar{Y}_t\big|^\beta+\big|u_t\big|^\beta+\big|\bar{u}_t\big|^\beta\right]\bigg\}\\
& + C\mathbb{E}\left(\int_0^T\int_{\mathcal{E}} \big|\delta f(t,e)\big|^2I_{\mathcal{O}}(t)N(d e,d t)\right)^{\frac{\beta}{2}}\leqslant C\epsilon^{\frac{\beta}{2}}.
\end{aligned}
$$
The proof is complete.
\end{proof}

\subsection{Proof of Lemma 3.2}
\begin{proof}
Note \eqref{variational equation of X1} and \eqref{variational equation of Y1}, by Proposition \ref{Lp-estimate of fully coupled FBSDEP}, we obtain
\begin{equation}\label{proof order of X1}
\begin{aligned}
&\mathbb{E}\Bigg\{\sup\limits_{0\leqslant t\leqslant T}\left[\big|X_t^1\big|^\beta+\big|Y_t^1\big|^\beta\right]+\left(\int_0^T \big\|Z_{(t,e)}^1\big\|^2 d t\right)^{\frac{\beta}{2}}+\left(\int_0^T\int_{\mathcal{E}}\big|\tilde{Z}_{(t,e)}^1\big|^2 N(d e,d t)\right)^{\frac{\beta}{2}}\Bigg\} \\
\leqslant &\ C\mathbb{E}\left\{\left(\int_0^T\big|b_z(t)\Delta^1(t)\big|I_{[\bar{t},\bar{t}+\epsilon]}(t)d t\right)^\beta\right.\\
&\qquad +\left(\int_0^T\Big(\big|g_z(t)\Delta^1(t)\big|+\big\|q_{(t,e)}\big\|\big|\delta\sigma(t,\Delta^1)\big|\Big)I_{[\bar{t},\bar{t}+\epsilon]}(t)d t\right)^\beta \\
&\qquad +\left(\int_0^T\Big(\big|\sigma_z(t)\Delta^1(t)\big|+\big|\delta\sigma(t,\Delta^1)\big| \Big)^2 I_{[\bar{t},\bar{t}+\epsilon]}(t) d t\right)^{\frac{\beta}{2}}\Bigg\}\\
\leqslant&\ C\mathbb{E}\Bigg\{\left(\int_{\bar{t}}^{\bar{t}+\epsilon}\left(1+\big|\bar{X}_t\big|+\big|\bar{Y}_t\big|+\big|u_t\big|+\big|\bar{u}_t\big|\right)d t\right)^\beta\\
&\qquad +\left(\int_{\bar{t}}^{\bar{t}+\epsilon}\left(1+\big|\bar{X}_t\big|+\big|\bar{Y}_t\big|+\big|u_t\big|+\big|\bar{u}_t\big|\right)^2d t\right)^{\frac{\beta}{2}}\Bigg\}\\
\leqslant &\ C\big(\epsilon^\beta+\epsilon^{\frac{\beta}{2}}\big)\mathbb{E}\bigg\{1+\sup\limits_{0\leqslant t\leqslant T}\left[\big|\bar{X}_t\big|^\beta
 +\big|\bar{Y}_t\big|^\beta+\big|u_t\big|^\beta+\big|\bar{u}_t\big|^\beta\right]\bigg\}\leqslant C\epsilon^{\frac{\beta}{2}}.
\end{aligned}
\end{equation}

Recall the notations used in Lemma \ref{estimates of X,Y,Z,tilde Z}, we also denote
\begin{equation*}
\begin{array}{c}
\hat{X}_t^\epsilon\triangleq X_t^\epsilon-\bar{X}_t-X_t^1,\quad \hat{Y}_t^\epsilon\triangleq Y_t^\epsilon-\bar{Y}_t-Y_t^1, \\[2ex]
\hat{Z}_{(t,e)}^\epsilon\triangleq Z_{(t,e)}^\epsilon-\bar{Z}_{(t,e)}-Z_{(t,e)}^1, \quad \hat{\tilde{Z}}_{(t,e)}^\epsilon\triangleq \tilde{Z}_{(t,e)}^\epsilon-\bar{\tilde{Z}}_{(t,e)}-\tilde{Z}_{(t,e)}^1, \\[2ex]
{\Theta}_t^{\Delta}\triangleq \big(\bar{X}_t,\bar{Y}_t,\bar{Z}_{(t,e)}+\Delta^1(t)I_{[\bar{t},\bar{t}+\epsilon]}(t),\bar{\tilde{Z}}_{(t,e)}\big).
\end{array}
\end{equation*}
Noting (\ref{variational equation FBSDEP}) and (\ref{variational equation of X1}), we get
$$
\begin{aligned}
&\quad \sigma\left(t,\Theta_t^\epsilon,u_t^\epsilon\right)-\sigma(t)-\delta\sigma(t,\Delta^1)I_{[\bar{t},\bar{t}+\epsilon]}(t) \\
&= \sigma\left(t,\Theta_t^\epsilon,u_t^\epsilon\right)-\sigma(t)-\sigma\big(t,\bar{X}_t,\bar{Y}_t,\bar{Z}_{(t,e)}+\Delta^1(t)I_{[\bar{t},\bar{t}+\epsilon]}(t),\bar{\tilde{Z}}_{(t,e)},u_t^\epsilon\big)+\sigma(t) \\
&= \sigma\left(t,\Theta_t^\epsilon,u_t^\epsilon\right)-\sigma\left(t,{\Theta}_t^{\Delta},u_t^\epsilon\right)
 =\tilde{\sigma}_x^\epsilon(t)\left(X_t^\epsilon-\bar{X}_t\right)+\tilde{\sigma}_y^\epsilon(t)\left(Y_t^\epsilon-\bar{Y}_t\right) \\
&\quad  +\tilde{\sigma}_z^\epsilon(t)\left(\int_{\mathcal{E}}\big(Z_{(t,e)}^\epsilon-\bar{Z}_{(t,e)}\big)\nu(d e)
 -\Delta^1(t)I_{[\bar{t},\bar{t}+\epsilon]}(t)\right)+\tilde{\sigma}_{\tilde{z}}^\epsilon(t)\int_{\mathcal{E}}\big(\tilde{Z}_{(t,e)}^\epsilon-\bar{\tilde{Z}}_{(t,e)}\big)\nu(d e),
\end{aligned}
$$
where
$$
\tilde{\sigma}_x^\epsilon(t)\triangleq \int_0^1\sigma_x\left(t,{\Theta}_t^{\Delta}+\theta\left({\Theta}_t^\epsilon-{\Theta}_t^{\Delta}\right)u_t^\epsilon\right)d\theta, \\
$$
and similarly for $\tilde{\sigma}_y^\epsilon(t),\tilde{\sigma}_z^\epsilon(t),\tilde{\sigma}_{\tilde{z}}^\epsilon(t)$.

Then, we can obtain the following equation with the definition in \eqref{definition of b}:
\begin{equation}\label{variational equation of Xepsilon}
\left\{\begin{aligned}
d\hat{X}_t^\epsilon&=\left\{\tilde{b}_x(t)\hat{X}_t^\epsilon+\tilde{b}_y(t)\hat{Y}_t^\epsilon+\int_\mathcal{E}\big(\tilde{b}_z(t)\hat{Z}_{(t,e)}^\epsilon
 +\tilde{b}_{\tilde{z}}(t)\hat{\tilde{Z}}_{(t,e)}^\epsilon\big)\nu(d e)+A_1^\epsilon(t)\right\}d t \\
&\quad +\left\{\tilde{\sigma}_x^\epsilon(t)\hat{X}_t^\epsilon+\tilde{\sigma}_y^\epsilon(t)\hat{Y}_t^\epsilon+\int_\mathcal{E}\big(\tilde{\sigma}_z^\epsilon(t)\hat{Z}_{(t,e)}^\epsilon
 +\tilde{\sigma}_{\tilde{z}}^\epsilon(t)\hat{\tilde{Z}}_{(t,e)}^\epsilon\big)\nu(d e)+B_1^\epsilon(t)\right\}d W_t \\
&\quad +\int_\mathcal{E}\left\{\tilde{f}_x(t,e)\hat{X}_t^\epsilon+\tilde{f}_y(t,e)\hat{Y}_t^\epsilon+\tilde{f}_{\tilde{z}}(t,e)\hat{\tilde{Z}}_{(t,e)}^\epsilon+C_1^\epsilon(t)\right\}\tilde{N}(d e,d t), \\
\hat{X}_0^\epsilon&=0,
\end{aligned}\right.
\end{equation}
where
\begin{equation*}
\begin{aligned}
A_1^\epsilon(t)&\triangleq \big(\tilde{b}_x(t)-b_x(t)\big)X_t^1+\big(\tilde{b}_y(t)-b_y(t)\big)Y_t^1+\big(\tilde{b}_z(t)-b_z(t)\big)\int_\mathcal{E}Z_{(t,e)}^1\nu(d e) \\
&\quad +\big(\tilde{b}_{\tilde{z}}(t)-b_{\tilde{z}}(t)\big)\int_\mathcal{E}\tilde{Z}_{(t,e)}^1\nu(d e)+b_z(t)\Delta^1(t)I_{[\bar{t},\bar{t}+\epsilon]}(t)+\delta b(t)I_{[\bar{t},\bar{t}+\epsilon]}(t), \\
B_1^\epsilon(t)&\triangleq \big(\tilde{\sigma}_x^\epsilon(t)-\sigma_x(t)\big)X_t^1+\big(\tilde{\sigma}_y^\epsilon(t)-\sigma_y(t)\big)Y_t^1\\
&\quad +\big(\tilde{\sigma}_z^\epsilon(t)-\sigma_z(t)\big)\int_\mathcal{E}K_1(t,e)\nu(d e)X_t^1+\big(\tilde{\sigma}_{\tilde{z}}^\epsilon(t)-\sigma_{\tilde{z}}(t)\big)\int_\mathcal{E}K_2(t,e)\nu(d e)X_t^1, \\
C_1^\epsilon(t,e)&\triangleq \big(\tilde{f}_x(t,e)-f_x(t,e)\big)X_t^1+\big(\tilde{f}_y(t,e)-f_y(t,e)\big)Y_t^1 \\
&\quad +\big(\tilde{f}_{\tilde{z}}(t,e)-f_{\tilde{z}}(t,e)\big)K_2(t,e)X_t^1+\delta f(t,e)I_{\mathcal{O}}(t).
\end{aligned}
\end{equation*}

Similarly, we obtain the following equation by \eqref{variational equation FBSDEP} and \eqref{variational equation of Y1}:
\begin{equation}\label{variational equation of Yepsilon}
\left\{\begin{aligned}
-d\hat{Y}_t^\epsilon&=\bigg\{\tilde{g}_x(t)\hat{X}_t^\epsilon+\tilde{g}_y(t)\hat{Y}_t^\epsilon+\int_\mathcal{E}\big(\tilde{g}_z(t)\hat{Z}_{(t,e)}^\epsilon
 +\tilde{g}_{\tilde{z}}(t)\hat{\tilde{Z}}_{(t,e)}^\epsilon\big)\nu(d e)+D_1^\epsilon(t)\bigg\}d t \\
&\quad -\int_\mathcal{E}\hat{Z}_{(t,e)}^\epsilon\nu(d e)d W_t-\int_\mathcal{E}\hat{\tilde{Z}}_{(t,e)}^\epsilon\tilde{N}(d e,d t), \\
\hat{Y}_0^\epsilon&= \tilde{\phi}_x(T)\hat{X}_T^\epsilon+E_1^\epsilon(T),
\end{aligned}\right.
\end{equation}
where
\begin{equation*}
\begin{aligned}
D_1^\epsilon(t)&\triangleq \big(\tilde{g}_x(t)-g_x(t)\big)X_t^1+\big(\tilde{g}_y(t)-g_y(t)\big)Y_t^1+\big(\tilde{g}_z(t)-g_z(t)\big)\int_\mathcal{E}Z_{(t,e)}^1\nu(d e) \\
&\qquad +\big(\tilde{g}_{\tilde{z}}(t)-g_{\tilde{z}}(t)\big)\int_\mathcal{E}\tilde{Z}_{(t,e)}^1\nu(d e)+\big(g_z(t)\Delta^1(t)+\delta g(t)+q_{(t,e)}\delta\sigma(t,\Delta^1)\big)I_{[\bar{t},\bar{t}+\epsilon]}(t), \\
E_1^\epsilon(T)&\triangleq \big(\tilde{\phi}_x(T)-\phi_x(\bar{X}_T)\big)X_T^1.
\end{aligned}
\end{equation*}
For \eqref{variational equation of Xepsilon} and \eqref{variational equation of Yepsilon}, applying Proposition \ref{Lp-estimate of fully coupled FBSDEP} again, we have
\begin{equation*}
\begin{aligned}
&\quad \mathbb{E}\Bigg\{\sup\limits_{0\leqslant t\leqslant T}\left[\big|\hat{X}_t^\epsilon\big|^2+\big|\hat{Y}_t^\epsilon\big|^2\right]
 +\int_0^T \big\|\hat{Z}_{(t,e)}^\epsilon\big\|^2 d t+\int_0^T\int_{\mathcal{E}}\big|\hat{\tilde{Z}}_{(t,e)}^\epsilon\big|^2 N(d e,d t)\Bigg\}\\
&\leqslant C\mathbb{E}\Bigg\{\big|E_1^\epsilon(T)\big|^2+\left(\int_0^T\big|A_1^\epsilon(t)\big|d t\right)^2+\left(\int_0^T\big|D_1^\epsilon(t)\big|d t\right)^2+\int_0^T\big|B_1^\epsilon(t)\big|^2 d t \\
&\qquad\quad +\int_0^T\int_{\mathcal{E}}\big|C_1^\epsilon(t,e)\big|^2 N(d e,d t)\Bigg\}\triangleq C\big(I_1+I_2+I_3+I_4+I_5\big).
\end{aligned}
\end{equation*}

Next, we prove that all $I_1,I_2,I_3,I_4,I_5\sim O\left(\epsilon^2\right)$. Firstly, by Lemma \ref{estimates of X,Y,Z,tilde Z} and \eqref{proof order of X1}, we get
$$
\begin{aligned}
&I_1=\mathbb{E}\big|E_1^\epsilon(T)\big|^2=\mathbb{E}\Big[\big|\tilde{\phi}_x(T)-\phi_x(\bar{X}_T)\big|^2\big|X_T^1\big|^2\Big]\\
=&\ \mathbb{E}\bigg[\left(\int_0^1\left|\phi_x(\theta\hat{X}_T+\bar{X}_T)-\phi_x(\bar{X}_T)\right|d \theta\right)^2\big|X_T^1\big|^2\bigg]
 \leqslant C\mathbb{E}\Big[\big|\hat{X}_T\big|^2\big|X_T^1\big|^2\Big]\\
\leqslant&\ C\Bigg(\mathbb{E}\bigg[\sup\limits_{0\leqslant t\leqslant T}\big|\hat{X}_t\big|^4\bigg]\Bigg)^{\frac{1}{2}}
 \Bigg(\mathbb{E}\bigg[\sup\limits_{0\leqslant t\leqslant T}\big|X_t^1\big|^4\bigg]\Bigg)^{\frac{1}{2}}\leqslant C\epsilon^2.
\end{aligned}
$$

Note that
$$
\begin{aligned}
&\ \big|\tilde{b}_x(t)-b_x(t)\big| \\
\leqslant& \int_0^1 \left|b_x\big(t,\theta\hat{X}_t+\bar{X}_t,\theta\hat{Y}_t+\bar{Y}_t,\theta\hat{Z}_{(t,e)}+\bar{Z}_{(t,e)},\theta\hat{\tilde{Z}}_{(t,e)}
 +\bar{\tilde{Z}}_{(t,e)},u_t^\epsilon\big)-b_x(t)\right|d \theta \\
\leqslant&\ C\left(\big|\hat{X}_t\big|+\big|\hat{Y}_t\big|+\big\|\hat{Z}_{(t,e)}\big\|+\big\|\hat{\tilde{Z}}_{(t,e)}\big\|\right)+\big|\delta b_x(t)I_{[\bar{t},\bar{t}+\epsilon]}(t)\big|,
\end{aligned}
$$
so
\begin{equation}
\begin{aligned}
&\ \mathbb{E}\left(\int_0^T \big|\big(\tilde{b}_x(t)-b_x(t)\big)X_t^1\big|d t\right)^2\\
\leqslant&\ \mathbb{E}\Bigg[\sup\limits_{0\leqslant t\leqslant T}\big|X_t^1\big|^2\left(\int_0^T \big|\tilde{b}_x(t)-b_x(t)\big|d t\right)^2\Bigg] \\
\leqslant&\ C\Bigg(\mathbb{E}\bigg[\sup\limits_{0\leqslant t\leqslant T}\big|X_t^1\big|^4\bigg]\Bigg)^{\frac{1}{2}}
 \Bigg(\mathbb{E}\left(\int_0^T\big|\tilde{b}_x(t)-b_x(t)\big| d t\right)^4\Bigg)^{\frac{1}{2}}\\
\leqslant&\ C\Bigg(\mathbb{E}\bigg[\sup\limits_{0\leqslant t\leqslant T}\big|X_t^1\big|^4\bigg]\Bigg)^{\frac{1}{2}}
 \Bigg(\mathbb{E}\left(\int_0^T\big|\tilde{b}_x(t)-b_x(t)\big|^2 d t\right)^2\Bigg)^{\frac{1}{2}}\\
\leqslant&\ C\epsilon\Bigg(\mathbb{E}\left(\int_0^T\left(\big|\hat{X}_t\big|^2+\big|\hat{Y}_t\big|^2+\big\|\hat{Z}_{(t,e)}\big\|^2
 +\big\|\hat{\tilde{Z}}_{(t,e)}\big\|^2+\big|\delta b_x(t)\big|^2I_{[\bar{t},\bar{t}+\epsilon]}(t)\right) d t\right)^2\Bigg)^{\frac{1}{2}}\\
\leqslant&\ C\epsilon\Bigg(\mathbb{E}\Bigg\{\sup\limits_{0\leqslant t\leqslant T}\Big[\big|\hat{X}_t\big|^4+\big|\hat{Y}_t\big|^4\Big]+\left(\int_0^T\big\|\hat{Z}_{(t,e)}\big\|^2 d t\right)^2\\
&\qquad +\left(\int_0^T\big\|\hat{\tilde{Z}}_{(t,e)}\big\|^2 d t\right)^2+\left(\int_0^T\big|\delta b_x(t)\big|^2I_{[\bar{t},\bar{t}+\epsilon]}(t) d t\right)^2\Bigg\}\Bigg)^{\frac{1}{2}}\\
\leqslant&\ C\epsilon \Bigg(\mathbb{E}\Bigg\{\sup\limits_{0\leqslant t\leqslant T}\Big[\big|\hat{X}_t\big|^4+\big|\hat{Y}_t\big|^4\Big]+\left(\int_0^T\big\|\hat{Z}_{(t,e)}\big\|^2 d t\right)^2\\
&\qquad +\left(\int_0^T\int_{\mathcal{E}}\big|\hat{\tilde{Z}}_{(t,e)}\big|^2 N(d e,d t)\right)^2+\left(\int_0^T\big|\delta b_x(t)\big|^2I_{[\bar{t},\bar{t}+\epsilon]}(t) d t\right)^2\Bigg\}\Bigg)^{\frac{1}{2}}
\leqslant C\epsilon^2.
\end{aligned}
 \label{proof order of b}
\end{equation}
The estimates of $\mathbb{E}\left(\int_0^T \big|\big(\tilde{b}_y(t)-b_y(t)\big)Y_t^1\big|d t\right)^2$, $\mathbb{E}\left(\int_0^T \big|\big(\tilde{b}_z(t)-b_z(t)\big)Z_{(t,e)}^1\big|d t\right)^2$ and\\ $\mathbb{E}\left(\int_0^T \big|\big(\tilde{b}_{\tilde{z}}(t)-b_{\tilde{z}}(t)\big)\tilde{Z}_{(t,e)}^1\big|d t\right)^2$ are similar.

Note $\Delta^1(t)=p_t\delta\sigma\left(t,\Delta^1\right)$,
\begin{equation}
\begin{aligned}
&\ \mathbb{E}\left(\int_0^T \big|b_z(t)\Delta^1(t)I_{[\bar{t},\bar{t}+\epsilon]}(t)\big|d t\right)^2
\leqslant C\mathbb{E}\left(\int_{\bar{t}}^{\bar{t}+\epsilon}\big|\delta\sigma\left(t,\Delta^1\right)\big|d t\right)^2\\
\leqslant&\ \mathbb{E}\bigg(\int_{\bar{t}}^{\bar{t}+\epsilon}\big(1+\big|\bar{X}_t\big|+\big|\bar{Y}_t\big|+\big|u_t\big|+\big|\bar{u}_t\big|\big)d t\bigg)^2 \\
\leqslant&\ C\epsilon^2\mathbb{E}\bigg\{1+\sup\limits_{0\leqslant t\leqslant T}\left[\big|\bar{X}_t\big|^2+\big|\bar{Y}_t\big|^2+\big|u_t\big|^2+\big|\bar{u}_t\big|^2\right]\bigg\}\leqslant C\epsilon^2,
\end{aligned}
 \label{proof order of Delta}
\end{equation}
as well as
\begin{equation}
\begin{aligned}
&\ \mathbb{E}\left(\int_0^T \big|\delta b(t)I_{[\bar{t},\bar{t}+\epsilon]}(t)\big|d t\right)^2\leqslant \epsilon\mathbb{E}\int_{\bar{t}}^{\bar{t}+\epsilon} \big|\delta b(t)\big|^2 d t\\
\leqslant&\ C\epsilon \mathbb{E}\int_{\bar{t}}^{\bar{t}+\epsilon}\Big(1+\big|\bar{X}_t\big|^2+\big|\bar{Y}_t\big|^2+\big|u_t\big|^2+\big|\bar{u}_t\big|^2\Big)d \bigg\} \\
\leqslant&\ C\epsilon^2\mathbb{E}\bigg\{1+\sup\limits_{0\leqslant t\leqslant T}\left[\big|\bar{X}_t\big|^2+\big|\bar{Y}_t\big|^2+\big|u_t\big|^2+\big|\bar{u}_t\big|^2\right]\bigg\}\leqslant C\epsilon^2,
\end{aligned}
 \label{proof order of deltab}
\end{equation}
then, we obtain
$$
\begin{aligned}
I_2\leqslant&\ C\mathbb{E}\left(\int_0^T \big|\big(\tilde{b}_x(t)-b_x(t)\big)X_t^1\big|d t\right)^2 +C\mathbb{E}\left(\int_0^T \big|\big(\tilde{b}_y(t)-b_y(t)\big)Y_t^1\big|d t\right)^2 \\
& +C\mathbb{E}\left(\int_0^T \big|\tilde{b}_z(t)-b_z(t)\big|\big\|Z_{(t,e)}^1\big\|d t\right)^2
 +C\mathbb{E}\left(\int_0^T \big|\tilde{b}_{\tilde{z}}(t)-b_{\tilde{z}}(t)\big|\big\|\tilde{Z}_{(t,e)}^1\big\|d t\right)^2\\
& +C\mathbb{E}\left(\int_0^T \big|b_z(t)\Delta^1(t)I_{[\bar{t},\bar{t}+\epsilon]}(t)\big|d t\right)^2
 +C\mathbb{E}\left(\int_0^T \big|\delta b(t)I_{[\bar{t},\bar{t}+\epsilon]}(t)\big|d t\right)^2 \leqslant C\epsilon^2.
\end{aligned}
$$
Similarly for $I_3\leqslant C\epsilon^2$. Next, we only prove the fourth term in $I_4$.

Note that $K_1(t,e)$ and $K_2(t,e)$ in \eqref{explicit solution of K1K2} are bounded by $\left(p_\cdot,q_{(\cdot,\cdot)},\tilde{q}_{(\cdot,\cdot)}\right)$ and the Lipschitz constants of $\sigma$ and $f$, so
$$
\begin{aligned}
&\mathbb{E}\int_0^T \left\|\big(\tilde{\sigma}_{\tilde{z}}^\epsilon(t)-\sigma_{\tilde{z}}(t)\big)K_2(t,e)X_t^1\right\|^2 d t
\leqslant C\mathbb{E}\bigg\{\sup\limits_{0\leqslant t\leqslant T}\big|X_t^1\big|^2\int_0^T \big|\tilde{\sigma}_{\tilde{z}}^\epsilon(t)-\sigma_{\tilde{z}}(t)\big|^2 d t\bigg\} \\
\leqslant&\ C\left(\mathbb{E}\bigg[\sup\limits_{0\leqslant t\leqslant T}\big|X_t^1\big|^4\bigg]\right)^{\frac{1}{2}}
 \left(\mathbb{E}\left(\int_0^T \big|\tilde{\sigma}_{\tilde{z}}^\epsilon(t)-\sigma_{\tilde{z}}(t)\big|^2 d t\right)^2\right)^{\frac{1}{2}}.
\end{aligned}
$$

We can get $\mathbb{E}\Big[\sup_{0\leqslant t\leqslant T}\big|X_t^1\big|^4\Big]\leqslant C\epsilon^2$ by \eqref{proof order of X1} and $\mathbb{E}\left(\int_0^T \big|\tilde{\sigma}_{\tilde{z}}^\epsilon(t)-\sigma_{\tilde{z}}(t)\big|^2 d t\right)^2\leqslant C\epsilon^2$ by the similar proofs in \eqref{proof order of b}, \eqref{proof order of Delta} and \eqref{proof order of deltab}. So we have $I_4 \sim O\left(\epsilon^2\right)$. By the similar proof and the same reason in Remark 3.1, $\delta f(t)I_{\mathcal{O}}(t)$ do not influence the order of variation, so $I_5 \sim O\left(\epsilon^2\right)$ and we obtain \eqref{order of X-X-X2}.

Similarly, we have
\begin{equation*}
\begin{aligned}
& \mathbb{E}\bigg\{\sup\limits_{0\leqslant t\leqslant T}\left[\big|\hat{X}_t^\epsilon\big|^4+\big|\hat{Y}_t^\epsilon\big|^4\right]+\left(\int_0^T \big\|\hat{Z}_{(t,e)}^\epsilon\big\|^2 d t\right)^2
 +\left(\int_0^T\int_{\mathcal{E}}\big|\hat{\tilde{Z}}_{(t,e)}^\epsilon\big|^2 N(d e,d t)\right)^2\bigg\}\\
\leqslant&\ C\mathbb{E}\Bigg\{\big|E_1^\epsilon(T)\big|^4+\left(\int_0^T\big|A_1^\epsilon(t)\big|d t\right)^4+\left(\int_0^T\big|D_1^\epsilon(t)\big|d t\right)^4 \\
&\qquad +\left(\int_0^T\big|B_1^\epsilon(t)\big|^2 d t\right)^2+\left(\int_0^T\int_{\mathcal{E}}\big|C_1^\epsilon(t,e)\big|^2 N(d e,d t)\right)^2\Bigg\}= o(\epsilon^2).
\end{aligned}
\end{equation*}
The proof is complete.
\end{proof}

\subsection{Proof of Lemma 3.3}
\begin{proof}
By Proposition \ref{Lp-estimate of fully coupled FBSDEP}, we have
\begin{equation}
\begin{aligned}
&\mathbb{E}\Bigg\{\sup\limits_{0\leqslant t\leqslant T}\left[\big|X_t^2\big|^\beta+\big|Y_t^2\big|^\beta\right]+\left(\int_0^T \big\|Z_{(t,e)}^2\big\|^2 d t\right)^{\frac{\beta}{2}}
 +\left(\int_0^T\int_{\mathcal{E}}\big|\tilde{Z}_{(t,e)}^2\big|^2 N(d e,d t)\right)^{\frac{\beta}{2}}\Bigg\} \\
\leqslant&\ C\mathbb{E}\Bigg\{\left(\int_0^T\left|\delta b\left(t,\Delta^1\right)I_{[\bar{t},\bar{t}+\epsilon]}(t)+\frac{1}{2}\tilde{\Xi}_t D^2b(t)\tilde{\Xi}_t^\top\right|d t\right)^\beta\\
&\qquad +\bigg(\int_0^T\bigg\|\frac{1}{2}\tilde{\Xi}_tD^2g(t)\tilde{\Xi}_t^\top+\Big(q_{(t,e)}\delta\sigma\left(t,\Delta^1\right)+\delta g\left(t,\Delta^1\right)\Big)I_{[\bar{t},\bar{t}+\epsilon]}(t)\bigg\|d t\bigg)^\beta \\
&\qquad +\bigg(\int_0^T\Big\|\Big(\delta\sigma_x\left(t,\Delta^1\right) X_t^1 +\delta\sigma_y\left(t,\Delta^1\right)Y_t^1+\delta\sigma_z\left(t,\Delta^1\right)K_1(t,e)X_t^1 \\
&\qquad\qquad\qquad +\delta\sigma_{\tilde{z}}\left(t,\Delta^1\right)K_2(t,e)X_t^1\Big)I_{[\bar{t},\bar{t}+\epsilon]}(t)+\frac{1}{2}\tilde{\Xi}_tD^2\sigma(t)\tilde{\Xi}_t^\top\Big\|^2 d t\bigg)^{\frac{\beta}{2}} \\
&\qquad +\bigg(\int_0^T\int_{\mathcal{E}}\left|\frac{1}{2}\tilde{\Xi}_tD^2f(t,e)\tilde{\Xi}_t^\top\right|^2 N(d e,d t)\bigg)^{\frac{\beta}{2}}\Bigg\}\\
\leqslant&\ C\mathbb{E}\Bigg\{\left(\int_0^T\left(\big|\delta b\left(t,\Delta^1\right)\big|I_{[\bar{t},\bar{t}+\epsilon]}(t)+\big|X_t^1\big|^2+\big|Y_t^1\big|^2\right)d t\right)^\beta\\
&\qquad +\bigg(\int_0^T\left(\big|\delta\sigma\left(t,\Delta^1\right)+\delta g\left(t,\Delta^1\right)\big|I_{[\bar{t},\bar{t}+\epsilon]}(t)+\big|X_t^1\big|^2+\big|Y_t^1\big|^2\right) d t\bigg)^\beta \\
&\qquad +\bigg(\int_0^T \bigg(\big|X_t^1\big|^4+\big|Y_t^1\big|^4+\big(\big|X_t^1\big|^2+\big|Y_t^1\big|^2\big)I_{[\bar{t},\bar{t}+\epsilon]}(t)\bigg)d t\bigg)^{\frac{\beta}{2}} \\
&\qquad  +\left(\int_0^T\int_{\mathcal{E}}\Big(\big|X_t^1\big|^4+\big|Y_t^1\big|^4\Big) N(d e,d t)\right)^{\frac{\beta}{2}}\Bigg\}\\
\leqslant&\ C\mathbb{E}\left(\int_{\bar{t}}^{\bar{t}+\epsilon}\Big(1+\big|\bar{X}_t\big|+\big|\bar{Y}_t\big|+\big|u_t\big|+\big|\bar{u}_t\big|\Big)d t\right)^\beta
 +C\mathbb{E}\bigg[\sup\limits_{0\leqslant t\leqslant T}\Big(\big|X_t^1\big|^{2\beta}+\big|Y_t^1\big|^{2\beta}\Big)\bigg] \\
\end{aligned}
\end{equation}
\begin{equation}
\begin{aligned}
& +C\epsilon^{\frac{\beta}{2}}\mathbb{E}\bigg[\sup\limits_{0\leqslant t\leqslant T}\Big(\big|X_t^1\big|^{\beta}+\big|Y_t^1\big|^{\beta}\Big)\bigg]= O(\epsilon^\beta)=o\left(\epsilon^{\frac{\beta}{2}}\right).
\end{aligned}
\end{equation}

Now, we prove the last estimate. Recall the notations used in Lemma \ref{estimates of X,Y,Z,tilde Z} and Lemma \ref{high order estimates}, we denote
\begin{equation*}
\begin{array}{c}
\hat{X}_t^*\triangleq X_t^\epsilon-\bar{X}_t-X_t^1-X_t^2,\quad \hat{Y}_t^*\triangleq Y_t^\epsilon-\bar{Y}_t-Y_t^1-Y_t^2, \\[1ex]
\hat{Z}_{(t,e)}^*\triangleq Z_{(t,e)}^\epsilon-\bar{Z}_{(t,e)}-Z_{(t,e)}^1-Z_{(t,e)}^2, \quad \hat{\tilde{Z}}_{(t,e)}^*\triangleq \tilde{Z}_{(t,e)}^\epsilon-\bar{\tilde{Z}}_{(t,e)}-\tilde{Z}_{(t,e)}^1-\tilde{Z}_{(t,e)}^2, \\[1ex]
\check{\Xi}_t\triangleq \hat{X}_t,\hat{Y}_t,\hat{Z}_{(t,e)}-\Delta^1(t)I_{[\bar{t},\bar{t}+\epsilon]}(t),\hat{\tilde{Z}}_{(t,e)},
\end{array}
\end{equation*}
and define
$$
\widetilde{D^2b^\epsilon}(t)\triangleq 2\int_0^1\int_0^1\theta D^2b(t,\Theta_t^\Delta+\lambda\theta(\Theta_t^\epsilon-\Theta_t^\Delta),u_t^\epsilon)d \theta d\lambda,
$$
and similarly for $\widetilde{D^2\sigma^\epsilon}(t),\widetilde{D^2f^\epsilon}(t,e),\widetilde{D^2g^\epsilon}(t)$ and $\tilde{\phi}_{x x}(T)$. Then we have
\begin{equation}\label{variational equation of X*}
\left\{\begin{aligned}
d\hat{X}_t^*&=\bigg\{b_x(t)\hat{X}_t^*+b_y(t)\hat{Y}_t^*+\int_\mathcal{E}\big(b_z(t)\hat{Z}_{(t,e)}^*+b_{\tilde{z}}(t)\hat{\tilde{Z}}_{(t,e)}^*\big)\nu(d e)+A_2^\epsilon(t)\bigg\}d t \\
&\quad +\bigg\{\sigma_x(t)\hat{X}_t^*+\sigma_y(t)\hat{Y}_t^*+\int_\mathcal{E}\big(\sigma_z(t)\hat{Z}_{(t,e)}^*+\sigma_{\tilde{z}}(t)\hat{\tilde{Z}}_{(t,e)}^*\big)\nu(d e)+B_2^\epsilon(t)\bigg\}d W_t \\
&\quad +\int_\mathcal{E}\left\{f_x(t,e)\hat{X}_t^*+f_y(t,e)\hat{Y}_t^*+f_{\tilde{z}}(t,e)\hat{\tilde{Z}}_{(t,e)}^*+C_2^\epsilon(t,e)\right\}\tilde{N}(d e,d t), \\
 \hat{X}_0^*&=0,
\end{aligned}\right.
\end{equation}
where
\begin{equation*}
\begin{aligned}
A_2^\epsilon(t)\triangleq &\bigg[\delta b_x(t,\Delta^1)\hat{X}_t+\delta b_y(t,\Delta^1)\hat{Y}_t+\delta b_z(t,\Delta^1)\left(\int_\mathcal{E}\hat{Z}_{(t,e)}\nu(d e)-\Delta^1(t)I_{[\bar{t},\bar{t}+\epsilon]}(t)\right) \\
& +\delta b_{\tilde{z}}(t,\Delta^1)\int_\mathcal{E}\hat{\tilde{Z}}_{(t,e)}\nu(d e)\bigg]I_{[\bar{t},\bar{t}+\epsilon]}(t) +\frac{1}{2}\check{\Xi}_t\widetilde{D^2b^\epsilon}(t)\check{\Xi}_t^\top-\frac{1}{2}\tilde{\Xi}_t D^2b(t)\tilde{\Xi}_t^\top, \\
B_2^\epsilon(t)\triangleq &\bigg[\delta \sigma_x(t,\Delta^1)\hat{X}_t^\epsilon+\delta \sigma_y(t,\Delta^1)\hat{Y}_t^\epsilon+\delta \sigma_z(t,\Delta^1)\int_\mathcal{E}\hat{Z}_{(t,e)}^\epsilon\nu(d e) \\
& +\delta \sigma_{\tilde{z}}(t,\Delta^1)\int_\mathcal{E}\hat{\tilde{Z}}_{(t,e)}^\epsilon\nu(d e)\bigg]I_{[\bar{t},\bar{t}+\epsilon]}(t) +\frac{1}{2}\check{\Xi}_t\widetilde{D^2\sigma^\epsilon}(t)\check{\Xi}_t^\top-\frac{1}{2}\tilde{\Xi}_t D^2\sigma(t)\tilde{\Xi}_t^\top, \\
C_2^\epsilon(t,e)\triangleq &\bigg[\delta f_x(t,\Delta^1,e)\hat{X}_t+\delta f_y(t,\Delta^1,e)\hat{Y}_t+\delta f_{\tilde{z}}(t,\Delta^1,e)\hat{\tilde{Z}}_{(t,e)}\bigg]I_{\mathcal{O}}(t)  \\
& +\frac{1}{2}\check{\Xi}_t\widetilde{D^2f^\epsilon}(t,e)\check{\Xi}_t^\top-\frac{1}{2}\tilde{\Xi}_t D^2f(t,e)\tilde{\Xi}_t^\top, \\
\end{aligned}
\end{equation*}
and
\begin{equation}\label{variational equation of Y*}
\left\{\begin{aligned}
-d\hat{Y}_t^*&=\bigg\{g_x(t)\hat{X}_t^*+g_y(t)\hat{Y}_t^*+\int_\mathcal{E}\big(g_z(t)\hat{Z}_{(t,e)}^*+g_{\tilde{z}}(t)\hat{\tilde{Z}}_{(t,e)}^*\big)\nu(d e)+D_2^\epsilon(t)\bigg\}d t \\
&\quad -\int_\mathcal{E}\hat{Z}_{(t,e)}^*\nu(d e)d W_t-\int_\mathcal{E}\hat{\tilde{Z}}_{(t,e)}^*\tilde{N}(d e,d t), \\
\hat{Y}_T^*&= {\phi}_x(\bar{X}_T)\hat{X}_T^*+E_2^\epsilon(T),
\end{aligned}\right.
\end{equation}
where
\begin{equation*}
\begin{aligned}
D_2^\epsilon(t)&\triangleq \bigg[\delta g_x(t,\Delta^1)\hat{X}_t+\delta g_y(t,\Delta^1)\hat{Y}_t+\delta g_z(t,\Delta^1)\left(\int_\mathcal{E}\hat{Z}_{(t,e)}\nu(d e)-\Delta^1(t)I_{[\bar{t},\bar{t}+\epsilon]}(t)\right) \\
&\qquad +\delta g_{\tilde{z}}(t,\Delta^1)\int_\mathcal{E}\hat{\tilde{Z}}_{(t,e)}\nu(d e)\bigg]I_{[\bar{t},\bar{t}+\epsilon]}(t)
 +\frac{1}{2}\check{\Xi}_t\widetilde{D^2g^\epsilon}(t)\check{\Xi}_t^\top-\frac{1}{2}\tilde{\Xi}_t D^2g(t)\tilde{\Xi}_t^\top, \\
E_2^\epsilon(T)&\triangleq \frac{1}{2}\tilde{\phi}_{x x}(T)(\hat{X}_T)^2-\frac{1}{2}\phi_{x x}(\bar{X}_T)(X_T^1)^2. \\
\end{aligned}
\end{equation*}

We introduce the following fully coupled FBSDEP:
\begin{equation}
\left\{\begin{aligned}
dh_t&=\left\{g_y(t)h_t+b_y(t)m_t+\int_\mathcal{E}\big(\sigma_y(t)n_{(t,e)}+f_y(t,e)\tilde{n}_{(t,e)}\big)\nu(d e)\right\}d t \\
&\quad +\left\{g_z(t)h_t+b_z(t)m_t+\int_\mathcal{E}\sigma_z(t)n_{(t,e)}\nu(d e)\right\}d W_t \\
&\quad +\int_\mathcal{E}\Big\{g_{\tilde{z}}(t)h_t+b_{\tilde{z}}(t)m_t+\sigma_{\tilde{z}}(t)n_{(t,e)}+f_{\tilde{z}}(t,e)\tilde{n}_{(t,e)}\Big\}\tilde{N}(d e,d t), \\
-dm_t&=\left\{g_x(t)h_t+b_x(t)m_t+\int_\mathcal{E}\big(\sigma_x(t)n_{(t,e)}+f_x(t,e)\tilde{n}_{(t,e)}\big)\nu(d e)\right\}d t \\
&\quad -\int_\mathcal{E}n_{(t,e)}\nu(d e)d W_t -\int_\mathcal{E}\tilde{n}_{(t,e)}\tilde{N}(d e,d t), \\
h_0&=1, \quad m_T=\phi_x(\bar{X}_T)h_T,
\end{aligned}\right.
\end{equation}
which has a unique solution $\big(h_\cdot,m_\cdot,n_{(\cdot,\cdot)},\tilde{n}_{(\cdot,\cdot)}\big)$ by Proposition \ref{Lp-estimate of fully coupled FBSDEP}. Applying It\^{o}'s formula to $m_t\hat{X}_t^*-h_t\hat{Y}_t^*$, we get
\begin{equation}
\begin{aligned}
\big|\hat{Y}_0^*\big|&= \left|\mathbb{E}\bigg\{h_T E_2^\epsilon(T)+\int_0^T \bigg(m_t A_2^\epsilon(t)+h_tD_2^\epsilon(t)+\int_\mathcal{E}\big(n_{(t,e)}B_2^\epsilon(t)
 +\tilde{n}_{(t,e)}C_2^\epsilon(t,e)\big)\nu(d e)\bigg)d t\bigg\}\right| \\
&\leqslant \mathbb{E}\big|h_T E_2^\epsilon(T)\big|+\mathbb{E}\int_0^T\big|m_t A_2^\epsilon(t)\big|d t+\mathbb{E}\int_0^T\big\|n_{(t,e)}B_2^\epsilon(t)\big\|d t\\
&\quad +\mathbb{E}\int_0^T\big\|\tilde{n}_{(t,e)}C_2^\epsilon(t,e)\big\|d t+\mathbb{E}\int_0^T\big|h_tD_2^\epsilon(t)\big|d t\triangleq I_6+I_7+I_8+I_9+I_{10}.
\end{aligned}
 \label{estimate with Y0}
\end{equation}

Next, we prove that all $I_6,I_7,I_8,I_9,I_{10}\sim o\left(\epsilon\right)$. Firstly, by Proposition \ref{Lp-estimate of fully coupled FBSDEP}, we have
$$
\mathbb{E}\Bigg\{\sup\limits_{0\leqslant t\leqslant T}\left[\big|h_t\big|^2+\big|m_t\big|^2\right]+\int_0^T\big\|n_{(t,e)}\big\|^2 d t+\int_0^T\int_{\mathcal{E}}\big|\tilde{n}_{(t,e)}\big|^2N(d e,d t)\Bigg\}\leqslant C.
$$
Note that
$$
\begin{aligned}
&\ E_2^\epsilon(T)=\frac{1}{2}\tilde{\phi}_{x x}(T)(\hat{X}_T)^2-\frac{1}{2}\phi_{x x}(\bar{X}_T)(X_T^1)^2 \\
\leqslant&\ C\left[\big(\tilde{\phi}_{x x}(T)-\phi_{x x}(\bar{X}_T)\big)(\hat{X}_T)^2+\phi_{x x}(\bar{X}_T)\big(\hat{X}_T-X_T^1\big)\big(\hat{X}_T+X_T^1\big)\right],
\end{aligned}
$$
and
$$
\begin{aligned}
&\ I_6=\mathbb{E}\big|h_T E_2^\epsilon(T)\big|\leqslant\left(\mathbb{E}\big|h_T\big|^2\right)^{\frac{1}{2}}\left(\mathbb{E}\big|E_2^\epsilon(T)\big|^2\right)^{\frac{1}{2}} \\
\leqslant&\ C\left(\mathbb{E}\bigg[\sup\limits_{0\leqslant t\leqslant T}\big|h_t\big|^2\bigg]\right)^{\frac{1}{2}}
 \left(\mathbb{E}\Big[\big|\tilde{\phi}_{x x}(T)-\phi_{x x}(\bar{X}_T)\big|^2\big|\hat{X}_T\big|^4\Big]
 +\mathbb{E}\Big[\big|\hat{X}_T^\epsilon\big|^2\big|\hat{X}_T+X_T^1\big|^2\Big]\right)^{\frac{1}{2}} \\
\leqslant&\ C\left(\mathbb{E}\Big[\big|\tilde{\phi}_{x x}(T)-\phi_{x x}(\bar{X}_T)\big|^2\big|\hat{X}_T\big|^4\Big]
 +\mathbb{E}\Big(\big|\hat{X}_T^\epsilon\big|^4\Big)^{\frac{1}{2}}\left[\mathbb{E}\Big(\big|\hat{X}_T\big|^4\Big)^{\frac{1}{2}}+\mathbb{E}\Big(\big|{X}_T^1\big|^4\Big)^{\frac{1}{2}}\right]\right)^{\frac{1}{2}}.
\end{aligned}
$$
So, we get $I_6\sim o(\epsilon)$. Since
$$
\begin{aligned}
I_7=&\ \mathbb{E}\int_0^T\big|m_t A_2^\epsilon(t)\big|d t\leqslant\mathbb{E}\bigg\{\sup\limits_{0\leqslant t\leqslant T}\big|m_t\big|\int_0^T\big| A_2^\epsilon(t)\big|d t\bigg\} \\
\leqslant&\ \left(\mathbb{E}\bigg[\sup\limits_{0\leqslant t\leqslant T}\big|m_t\big|^2\bigg]\right)^{\frac{1}{2}}\left(\mathbb{E}\left(\int_0^T\big| A_2^\epsilon(t)\big|d t\right)^2\right)^{\frac{1}{2}},
\end{aligned}
$$
we only need to prove $\mathbb{E}\int_0^T\big| A_2^\epsilon(t)\big|^2d t\sim o(\epsilon^2)$. Note that
$$
\begin{aligned}
&\ \mathbb{E}\left(\int_0^T\left|\delta b_z(t,\Delta^1)\left(\int_{\mathcal{E}}\hat{Z}_{(t,e)}\nu(d e)-\Delta^1(t)I_{[\bar{t},\bar{t}+\epsilon]}(t)\right)\right|I_{[\bar{t},\bar{t}+\epsilon]}(t)d t\right)^2 \\
\leqslant&\ \mathbb{E}\left(\int_{\bar{t}}^{\bar{t}+\epsilon}\big|\delta b_z(t,\Delta^1)\big|\left(\big\|\hat{Z}_{(t,e)}^\epsilon\big\|+\big\|K_1(t,e)X_t^1\big\|\right)d t\right)^2 \\
\leqslant&\ C\mathbb{E}\left(\int_{\bar{t}}^{\bar{t}+\epsilon}\big\|\hat{Z}_{(t,e)}^\epsilon\big\|d t\right)^2
 +C\mathbb{E}\left[\sup\limits_{0\leqslant t\leqslant T}\big|X_t^1\big|^2\left(\int_{\bar{t}}^{\bar{t}+\epsilon}\big|\delta b_z(t,\Delta^1)\big|d t\right)^2\right] \\
\leqslant&\ C\epsilon\mathbb{E}\int_0^T\big\|\hat{Z}_{(t,e)}^\epsilon\big\|^2d t +C\epsilon^2\mathbb{E}\left[\sup\limits_{0\leqslant t\leqslant T}\big|X_t^1\big|^2\right]=o(\epsilon^2),
\end{aligned}
$$
and
\begin{equation*}
\begin{aligned}
&\ \mathbb{E}\left(\int_0^T \Big\|\tilde{b}_{z z}^\epsilon(t)\big(\hat{Z}_{(t,e)}-\Delta^1(t)I_{[\bar{t},\bar{t}+\epsilon]}(t)\big)^2-b_{z z}(t)K_1(t,e)^2(X_t^1)^2\Big\| d t\right)^2 \\
\leqslant&\ C\mathbb{E}\left(\int_0^T \Big\|\tilde{b}_{z z}^\epsilon(t)\hat{Z}_{(t,e)}^\epsilon\big(\hat{Z}_{(t,e)}-\Delta^1(t)I_{[\bar{t},\bar{t}+\epsilon]}(t)+K_1(t,e)X_t^1\big)\Big\|d t\right)^2 \\
& +C\mathbb{E}\left(\int_0^T \Big\|\big(\tilde{b}_{z z}^\epsilon(t)-b_{z z}(t)\big)K_1(t,e)^2(X_t^1)^2\Big\|d t\right)^2 \\
\leqslant&\ C\mathbb{E}\left[\int_0^T \big\|\hat{Z}_{(t,e)}^\epsilon\big\|^2d t\int_0^T \Big\|\hat{Z}_{(t,e)}-\Delta^1(t)I_{[\bar{t},\bar{t}+\epsilon]}(t)+K_1(t,e)X_t^1\Big\|^2d t\right] \\
& +C\mathbb{E}\left[\sup\limits_{0\leqslant t\leqslant T}\big|X_t^1\big|^4\left(\int_0^T \big|\tilde{b}_{z z}^\epsilon(t)-b_{z z}(t)\big|d t\right)^2\right] = o(\epsilon^2).
\end{aligned}
\end{equation*}
Similarly for other estimates in $\mathbb{E}\int_0^T\big| A_2^\epsilon(t)\big|^2d t$, and we obtain $I_7\sim o(\epsilon)$.

For $I_8$, we have
\begin{equation} \label{problem estimate}
\begin{aligned}
&\ \mathbb{E}\int_0^T \Big\|n_{(t,e)}\delta \sigma_{\tilde{z}}(t,\Delta^1)\hat{\tilde{Z}}_{(t,e)}^\epsilon I_{[\bar{t},\bar{t}+\epsilon]}(t)\Big\|d t
\leqslant C\mathbb{E}\int_{\bar{t}}^{\bar{t}+\epsilon} \big\|n_{(t,e)}\hat{\tilde{Z}}_{(t,e)}^\epsilon\big\| d t \\
\leqslant&\ C\left(\mathbb{E}\int_{\bar{t}}^{\bar{t}+\epsilon} \big\|n_{(t,e)}\big\|^2d t\right)^{\frac{1}{2}}
 \left(\mathbb{E}\int_{\bar{t}}^{\bar{t}+\epsilon} \big\|\hat{\tilde{Z}}_{(t,e)}^\epsilon \big\|^2d t\right)^{\frac{1}{2}} \\
\leqslant&\ C\left(\mathbb{E}\int_{\bar{t}}^{\bar{t}+\epsilon} \big\|n_{(t,e)}\big\|^2d t\right)^{\frac{1}{2}}
 \left(\mathbb{E}\int_0^T \big\|\hat{\tilde{Z}}_{(t,e)}^\epsilon \big\|^2d t\right)^{\frac{1}{2}} \\
\leqslant&\ C\left(\mathbb{E}\int_{\bar{t}}^{\bar{t}+\epsilon} \big\|n_{(t,e)}\big\|^2d t\right)^{\frac{1}{2}}
 \left(\mathbb{E}\int_0^T \int_{\mathcal{E}}\big|\hat{\tilde{Z}}_{(t,e)}^\epsilon \big|^2 N(d e,d t)\right)^{\frac{1}{2}}=o(\epsilon).
\end{aligned}
\end{equation}
Then noting that
$$
\begin{aligned}
&\mathbb{E}\int_0^T\big\|n_{(t,e)}\big\|\left\|\tilde{\sigma}_{\tilde{z} \tilde{z}}^\epsilon(t)\big(\hat{\tilde{Z}}_{(t,e)}\big)^2-\sigma_{\tilde{z} \tilde{z}}(t) K_2(t,e)^2 (X_t^1)^2\right\| d t  \\
\leqslant &\ \mathbb{E}\int_0^T\left\|n_{(t,e)}\tilde{\sigma}_{\tilde{z} \tilde{z}}^\epsilon(t)\big(\hat{\tilde{Z}}_{(t,e)}+K_2(t,e) X_t^1\big) \hat{\tilde{Z}}_{(t,e)}^\epsilon\right\| d t \\
& +\mathbb{E}\int_0^T\Big\|n_{(t,e)}\big(\tilde{\sigma}_{\tilde{z} \tilde{z}}^\epsilon(t)-\sigma_{\tilde{z} \tilde{z}}(t)\big) K_2(t,e)^2 (X_t^1)^2 \Big\|d t \\
\leq &\ \mathbb{E}\int_0^T\Big\|n_{(t,e)}\tilde{\sigma}_{\tilde{z} \tilde{z}}^\epsilon(t)\hat{\tilde{Z}}_{(t,e)} \hat{\tilde{Z}}_{(t,e)}^\epsilon \Big\|d t
+\mathbb{E}\int_0^T\Big\|n_{(t,e)}\tilde{\sigma}_{\tilde{z} \tilde{z}}^\epsilon(t) K_2(t,e) X_t^1 \hat{\tilde{Z}}_{(t,e)}^\epsilon \Big\|d t+o(\epsilon) \\
\leqslant &\ \mathbb{E}\int_0^T\Big\|n_{(t,e)}\tilde{\sigma}_{\tilde{z} \tilde{z}}^\epsilon(t)\hat{\tilde{Z}}_{(t,e)} \hat{\tilde{Z}}_{(t,e)}^\epsilon \Big\|d t
+C \mathbb{E}\left[\sup _{t \in[0, T]}\left|X_t^1\right| \int_0^T\left\|n_{(t,e)}\hat{\tilde{Z}}_{(t,e)}^\epsilon\right\| d t\right]+o(\epsilon) \\
\leqslant &\ \mathbb{E}\int_0^T\Big\|n_{(t,e)}\tilde{\sigma}_{\tilde{z} \tilde{z}}^\epsilon(t)\hat{\tilde{Z}}_{(t,e)} \hat{\tilde{Z}}_{(t,e)}^\epsilon \Big\|d t+o(\epsilon),\\
\end{aligned}
$$
as well as
\begin{equation*}
\begin{aligned}
&\mathbb{E}\int_0^T\big\|n_{(t,e)}\big\|\left\|\tilde{\sigma}_{\tilde{z} z}^\epsilon(t)\big(\hat{Z}_{(t,e)}-\Delta^1(t)I_{[\bar{t},\bar{t}+\epsilon]}(t)\big)
 \hat{\tilde{Z}}_{(t,e)}-\sigma_{\tilde{z} z}(t) K_1(t,e) K_2(t,e)(X_t^1)^2\right\| d t \\
\leqslant &\ \mathbb{E}\int_0^T\big\|n_{(t,e)}\big\|\left\|\tilde{\sigma}_{\tilde{z} z}^\epsilon(t)\big(\hat{Z}_{(t,e)}^\epsilon +K_1(t,e)X_t^1\big)
 \hat{\tilde{Z}}_{(t,e)}-\sigma_{\tilde{z} z}(t) K_1(t,e) K_2(t,e)(X_t^1)^2\right\| d t\\
\leqslant &\ \mathbb{E}\int_0^T\Big\|n_{(t,e)}\tilde{\sigma}_{\tilde{z} z}^\epsilon(t)\big(\hat{Z}_{(t,e)}^\epsilon +K_1(t,e)X_t^1\big)\hat{\tilde{Z}}_{(t,e)}^\epsilon \Big\|d t \\
&\ +\mathbb{E}\int_0^T\big\|n_{(t,e)}\big\| \Big\|\tilde{\sigma}_{\tilde{z} z}^\epsilon(t)\big(\hat{Z}_{(t,e)}^\epsilon +K_1(t,e)X_t^1\big)K_2(t,e)X_t^1
 -\sigma_{\tilde{z} z}(t) K_1(t,e) K_2(t,e)(X_t^1)^2 \Big\|d t\\
\leqslant &\ \mathbb{E}\int_0^T\Big\|n_{(t,e)}\tilde{\sigma}_{\tilde{z} z}^\epsilon(t)\big(\hat{Z}_{(t,e)}^\epsilon +K_1(t,e)X_t^1\big)\hat{\tilde{Z}}_{(t,e)}^\epsilon \Big\|d t
 +C \mathbb{E}\left[\sup _{t \in[0, T]}\left|X_t^1\right| \int_0^T\big\|n_{(t,e)}\hat{Z}_{(t,e)}^\epsilon\big\| d t\right] \\
& +\mathbb{E}\int_0^T\Big\|n_{(t,e)}\big(\tilde{\sigma}_{\tilde{z} z}^\epsilon(t)-\sigma_{\tilde{z} z}(t)\big) K_1(t,e) K_2(t,e) (X_t^1)^2 \Big\|d t \\
\leqslant &\ \mathbb{E}\int_0^T\Big\|n_{(t,e)}\tilde{\sigma}_{\tilde{z} z}^\epsilon(t)\big(\hat{Z}_{(t,e)}^\epsilon +K_1(t,e)X_t^1\big)\hat{\tilde{Z}}_{(t,e)}^\epsilon \Big\|d t+o(\epsilon),
\end{aligned}
\end{equation*}
we can obtain
$$
\begin{aligned}
&\mathbb{E} \int_0^T\big\|n_{(t,e)}\big\|\left\|\tilde{\sigma}_{\tilde{z} \tilde{z}}^\epsilon(t)\big(\hat{\tilde{Z}}_{(t,e)}\big)^2-\sigma_{\tilde{z} \tilde{z}}(t) K_2(t,e)^2 (X_t^1)^2\right\| d t \\
&+\mathbb{E} \int_0^T\big\|n_{(t,e)}\big\|\left\|\tilde{\sigma}_{\tilde{z} z}^\epsilon(t)\big(\hat{Z}_{(t,e)}-\Delta^1(t)I_{[\bar{t},\bar{t}+\epsilon]}(t)\big)\hat{\tilde{Z}}_{(t,e)}
 -\sigma_{\tilde{z} z}(t) K_1(t,e) K_2(t,e)(X_t^1)^2\right\| d t \\
\leqslant&\ \mathbb{E}\int_0^T\Big\|n_{(t,e)}\tilde{\sigma}_{\tilde{z} \tilde{z}}^\epsilon(t)\hat{\tilde{Z}}_{(t,e)} \hat{\tilde{Z}}_{(t,e)}^\epsilon
 +n_{(t,e)}\tilde{\sigma}_{\tilde{z} z}^\epsilon(t)\big(\hat{Z}_{(t,e)}^\epsilon +K_1(t,e)X_t^1\big)\hat{\tilde{Z}}_{(t,e)}^\epsilon \Big\|d t +o(\epsilon) \\
\leqslant&\ \mathbb{E}\int_0^T \big\|n_{(t,e)}\big\|\left\|2\int_0^1\big(\sigma_{\tilde{z}}(t,\Theta_t^\Delta+\theta\left(\Theta_t^\epsilon-\Theta_t^\Delta\right),u_t^\epsilon,e)
 -\sigma_{\tilde{z}}(t,\Theta_t^\Delta ,u_t^\epsilon,e)\big)d \theta \right\|\big\|\hat{\tilde{Z}}_{(t,e)}^\epsilon\big\| d t \\
&\ +\mathbb{E}\int_0^T \big\|n_{(t,e)}\big\|\big\|\tilde{\sigma}_{\tilde{z} x}^\epsilon(t)\hat{X}_{(t,e)}
 +\tilde{\sigma}_{\tilde{z} y}^\epsilon(t)\hat{Y}_{(t,e)}\big\|\big\|\hat{\tilde{Z}}_{(t,e)}^\epsilon\big\| d t +o(\epsilon)= o(\epsilon).
\end{aligned}
$$
By similar proofs in \cite{HuJiXue2018}, we can get the other estimates, so we have $I_8\sim o(\epsilon)$. The estimate of $I_9$ is similar to $I_8$ and $I_{1 0}$ is the same as $I_7$, therefore all the estimates in \eqref{estimate with Y0} have been derived. Finally, we obtain
\begin{equation*}
Y_0^\epsilon-\bar{Y}_0-Y_0^1-Y_0^2=o(\epsilon).
\end{equation*}
The proof is complete.
\end{proof}

\end{document}